\documentclass[10pt,a4paper]{article}
\usepackage{amsmath}
\usepackage[arrow, matrix, curve]{xy}
\usepackage{amssymb}
\usepackage{a4wide}
\usepackage{amsthm}
\usepackage{url}
\usepackage{tikz}
\usepackage{latexsym}
\newtheorem{thm}{Theorem}[section]
\newtheorem{lemma}[thm]{Lemma}
\newtheorem{cor}[thm]{Corollary}

\theoremstyle{definition}
\newtheorem{definition}[thm]{Definition}
\theoremstyle{remark}
\newtheorem{remark}[thm]{Remark}

\DeclareMathOperator{\PGL}{PGL}
\DeclareMathOperator{\psl}{PSL}
\DeclareMathOperator{\SL}{SL}
\DeclareMathOperator{\GL}{GL}
\DeclareMathOperator{\Nrd}{Nrd}

\DeclareMathOperator{\aut}{Aut}
\DeclareMathOperator{\HQ}{\bf H}
\newcommand{\CC}{\mathbb C}
\newcommand{\NN}{\mathbb N}
\newcommand{\PP}{\mathbb P}
\newcommand{\QQ}{\mathbb Q}
\newcommand{\RR}{\mathbb R}
\newcommand{\ZZ}{\mathbb Z}
\newcommand{\HH}{\mathbb H}
\newcommand{\FF}{\mathbb F}
\newcommand{\BB}{\mathbb B}
\newcommand{\aaa}{\mathfrak a}

\newcommand{\pP}{\mathfrak p}

\newcommand{\OO}{\mathcal O}

\newcommand{\cc}{\mathcal C}
\newcommand{\oo}{{\scriptstyle{\mathcal O}}}
\begin{document}
\title{Fake quadrics from irreducible lattices acting on the product of upper half planes}
\author{Amir D\v{z}ambi\'c\footnote{ Institut f\"ur Mathematik, Goethe-Universit\"at Frankfurt a.M., Robert-Mayer-Str.~6-8, 60325 Frankfurt am Main. Email: dzambic@math.uni-frankfurt.de}
}
\date{}
\maketitle
\begin{abstract}
In the present article, we provide examples of fake quadrics, that is, minimal complex surfaces of general type with the same numerical invariants as the smooth quadric in $\PP ^3$ which are quotients of the bidisc by an irreducible lattice of automorphisms. Moreover we list classes of arithmetic lattices over a real quadratic number field which define a fake quadric and give general results towards a classification of all such fake quadrics.  
\end{abstract}

\section{Introduction}
A minimal smooth projective algebraic surface of general type over $\CC$ is called a \textbf{fake quadric} if it has the following numerical invariants:
$$ q=p_g=0,\ c_2=4,\ c_1^2=8.$$
Here, $q$ denotes the irregularity, $p_g$ the geometric genus, $c_2=e$ the Euler number and $c_1^2$ the selfintersection number of the canonical divisor. Fake quadrics are motivated by the study of \textbf{fake projective planes}, that is, surfaces of general type with the same Betti numbers as $\PP^2(\CC)$, or equivalently surfaces of general type with $q=p_g=0, c_2=3, c_1^2=9$. Fake quadrics are exactly those minimal surfaces of general type with the same Betti numbers as the smooth quadric ($\cong \PP^1\times \PP^1$) in $\PP^3$, a fact which motivates the name (see, for instance, \cite{HirzebruchWerke1}, p.~780). By Yau's theorem, every fake projective plane is a quotient $\Gamma\backslash \BB_2$ of the two-dimensional complex ball $\BB_2$ by a cocompact and torsion-free discrete group $\Gamma$ of holomorphic automorphisms of $\BB_2$. Using the fact, proved by S.-K.~Yeung and B.~Klingler, that such $\Gamma$ is arithmetic, in their famous work \cite{PrasadYeung07} (see also \cite{PrasadYeung10}), G.~Prasad and S.-K.~Yeung gave a classification of all fake projective planes, which is now completed by the work of Cartwright and Steger, see \cite{CartwrightSteger10}. In the case of fake quadrics we are far away from a complete picture. F.~Hirzebruch asked several still open questions about fake quadrics, which are motivated by his problem of determining all the complex structures on the product $S^2\times S^2$ of spheres (\cite{HirzebruchWerke1}, p.779). Biggest challenge seems to be the question on the universal covering of a fake quadric. All examples of minimal fake quadrics have bidisc, or equivalently the product $\HH\times \HH$ of two copies of the upper half plane, as the universal covering. But, we do not know if this is true in general. One version of Hirzebruch's question is, if there are fake quadrics with a finite or even trivial fundamental group. To be more precise, we essentialy only know two methods of construction of fake quadrics (both leading to surfaces uniformized by the bidisc).\\ One method, which goes back to M.~Kuga is to take an irreducible cocompact lattice $\Gamma$ in $\psl_2(\RR)\times \psl_2(\RR)$ 
which acts freely on $\HH\times \HH$ and is of covolume 1 with respect to a suitably normalized volume form. A theorem of Matsushima and Shimura ensures that the quotient $\Gamma\backslash \HH\times \HH$, which is a minimal surface of general type, has the desired numerical invariants. Based on Kuga's example, I.~Shavel constructed further examples in \cite{Shavel78}. There the author takes irreducible lattices which are defined as norm-1 groups of maximal orders in quaternion algebras over a real quadratic field with class number one. More examples of lattices of fake quadrics defined in quaternion algebras over quadratic fields are given by Otsubo (see \cite{Otsubo85}). All possible examples of norm-1 groups as lattices of fake quadrics have been found by K.~Takeuchi (see \cite{Takeuchi90}). The other method is due to Beauville and is as follows: take two curves $C_1$ and $C_2$ of genus $g_1>1$ and $g_2>1$ respectively, and let $G$ be a finite group of order $(g_1-1)(g_2-1)$ acting on $C_1$ and $C_2$ such that $C_i/G\cong \PP^1(\CC)$ for both $i=1,2$. Assume that $G$ acts on $C_1\times C_2$ without fixed points. Then, $(C_1\times C_2)/G$ is a fake quadric. In \cite{BCG08}, I.~Bauer, F.~Catanese and F~Grunewald started a systematic study of surfaces which are finitely covered by a product of two curves, so-called surfaces isogenous to a higher product, and in fact they classified all fake quadrics isogenous to a higher product.\\

The aim of the present paper is to start a classification of fake quadrics which are quotients of $\HH\times \HH$ by irreducible lattices in $\aut(\HH)\times \aut(\HH)$, the group of holomorphic factor preserving automorphisms of $\HH\times \HH$ which is isomorphic to $\psl_2(\RR)\times \psl_2(\RR)$. Note that by a theorem of Margulis, all such lattices are arithmetic and can therefore be explicitely described in terms of quaternion algebras over number fields. Moreover, there are explicit formulas expressing the covolume of these lattices (with respect to a suitably normalized Haar measure) in terms involving the arithmetic of the quaternion algebra. Having these informations, one can copy the approach of \cite{PrasadYeung07} to fake projective planes, which roughly goes as follows:\\
Let $\Gamma$ be an irreducible lattice in $\psl_2(\RR)\times \psl_2(\RR)$ and let $\chi(\Gamma)$ denote a suitable rational multiple of the Euler-Poincar\'e characteristic of $\Gamma$ (see Section \ref{fakequadricsandarithmetic} for the precise definition). The invariant $\chi$ is defined in such a way that for torsion-free lattices $\Gamma$, $\chi(\Gamma)=\chi(\Gamma\backslash \HH\times \HH)$ is the Euler-Poincar\'e characteristic of the structure sheaf of the associated complex surface. Moreover, one knows that $\Gamma\backslash \HH\times \HH$ is a fake quadric if and only if $\Gamma$ is torsion-free cocompact lattice with $\chi(\Gamma)=1$. Since every lattice $\Gamma$ is contained in a maximal lattice $\Delta$ with finite index and since $\chi(\Gamma)=[\Delta:\Gamma]\cdot \chi(\Delta)$, we see that if a maximal lattice $\Delta$ contains a lattice $\Gamma$ leading to a fake quadric one has $\chi(\Delta)\leq 1$ and $1/\chi(\Delta)$ is a positive integer. Therefore we can proceed as follows:\\
\begin{itemize}
\item List all maximal irreducible lattices $\Delta< \psl_2(\RR)\times \psl_2(\RR)$ with $\chi(\Delta)\leq 1$ and $\frac{1}{\chi(\Delta)}\in \NN$,
\item Find torsion-free subgroups $\Gamma<\Delta$ (if there are any) of index $\frac{1}{\chi(\Delta)}$ in $\Delta$.
\end{itemize}

For every irreducible lattice $\Gamma<\psl_2(\RR)\times \psl_2(\RR)$ there exists a totally real number field $k$ and a quaternion algebra $B$ over $k$ which is unramified exactly at two infinite places, such that $\Gamma$ is commensurable with the norm-1 group of a maximal order in $B$ modulo its center. The commensurability class (in wide sense) $\mathcal C(k,B)$ of such $\Gamma$ depends only on $k$ and $B$. By a theorem of Borel (\cite{Borel81}, Theorem 8.2) there are only finitely many conjugacy classes of irreducible lattices $\Gamma$ such that $\chi(\Gamma)\leq 1$.\\ 

Following ideas from \cite{PrasadYeung07} and \cite{ChinburgFriedman86} in our first main result we give explicit bounds for the degree of $k$ and some other invariants related to $\mathcal C(k,B)$ restricting commensurability classes which contain lattices of fake quadrics. Then, we examine more closely the simplest case where $k$ is a real quadratic field. Here we list all commensurability classes which contain lattices of fake quadrics. More precisely, we identify classes $[k, B, \OO, S, I]$, i.e. tuples where $k$ is a real quadratic field, $B$ is a totally indefinte quaternion algebra over $k$, $\OO$ a maximal order of $B$, $S$ a finite set of places and $I\in \NN$, consisting, by definition, of all irreducible lattices contained in a maximal lattice $\Gamma_{S,\OO}$ (see Section \ref{maximalandvol} for the definition) of index $I$ which define a fake quadric. We obtain 39 classes of lattices of fake quadrics defined over a quadratic field. Moreover, we give examples of lattices of fake quadrics over quadratic fields, thus showing that almost all classes are certainly non-empty. For the two classes $[\QQ(\sqrt{13}),v_3v'_3,\emptyset,12]$ and $[\QQ(\sqrt{17}),v_2v'_2, \emptyset,24]$ we were unable to produce an example of a lattice which belongs to that particular class. So, these classes may in fact be empty. It should be noticed that dividing fake quadrics in classes as above is not the finest possible classification since a class may contain several conjugacy classes of lattices, which all lead to non-isomorphic fake quadrics. The classification up to isomorphism is a still open question.\\

In principle, the method, which we applied in the case of a quadratic field can be copied in the case of totally real fields of higher degree. This would lead to a list of classes of all fake quadrics defined by irreducible lattices of $\psl_2(\RR)\times \psl_2(\RR)$. But, as the required invariants of the fields become more complicated with increasing degree, one needs to involve much more computer based calculations. We intend to analyze these remaining cases as far as possible (see also Remark \ref{praktibilitaet}).\\

To mention an application, let us remark that fake quadrics can also be used to construct surfaces of general type with interesting numerical invariants. In \cite{AX}, we studied quotients of fake quadrics by finite groups of automorphisms and obtained new surfaces of general type with geometric genus zero. This approach is in analogy with the study of quotients of fake projective planes due to Keum. \\

\noindent {\bf Acknowledgment:} I would like to thank Gopal Prasad and Sai-Kee Yeung for their remarks and helpful comments on fake quadrics. I especially thank Sai-Kee Yeung for pointing out a gap in the earlier version of the paper.

\section{Preliminaries}
In this section we would like to give a more detailed description of irreducible lattices in $\aut(\HH)\times\aut(\HH)$ and to present the technical tools which are used to find those lattices leading to fake quadrics.

\subsection{Fake quadrics and arithmetic lattices}
\label{fakequadricsandarithmetic}
Let $k\neq \QQ$ be a totally real number field of degree $[k:\QQ]=n\geq 2$ and and let $V_k$ (resp.~$V_k^{\infty}$ and $V_k^{f}$) denote the set of all places of $k$ (resp.~set of all infinite and the set of all finite places of $k$). Let $B$ be a quaternion algebra over $k$, which is ramified exactly at $n-2$ infinite places of $k$. Recall that, by definition, $B$ is ramified at a place $v\in V_k$ if $B_v:=B\otimes_k k_v$ is a division algebra over $k_v$ (which is in fact unique up to isomorphism) where $k_v$ denotes the completion of $k$ at $v$. From this we get
\begin{equation}
\label{reell}
B\otimes_{\QQ} \RR\cong \prod_{v\in V^{\infty}_k} B_v\cong M_2(\RR)\times M_2(\RR)\times \HQ_{\RR}^{n-2}  
\end{equation}
where $\HQ_{\RR}$ denotes the algebra of real Hamiltonian quaternions (unique division quaternion algebra over $\RR$).\\
Let $B^1=\{x\in B\mid \Nrd(x)=1\}$ be the group of all elements in $B$ with reduced norm $1$. Using the above isomorphism (\ref{reell}), $B^1$ is diagonally embedded in $\SL_2(\RR)\times \SL_2(\RR)\times \left(\HQ^1_{\RR}\right)^{n-2}$. Let $\OO$ be a maximal order in $B$ and $\OO^1=\{x\in \OO\mid \Nrd(x)=1\}$. In the same way as $B^1$, $\OO^1$ is diagonally embedded in $\SL_2(\RR)\times \SL_2(\RR)\times \left(\HQ^1_{\RR}\right)^{n-2}$ and additionally $\OO^1$ is a discrete subgroup there. It remains discrete subgroup also after the projection onto the product of first two factors $\SL_2(\RR)\times \SL_2(\RR)$. Moreover, the image of $\OO^1$ under a further projection onto one of the factors $\SL_2(\RR)$ is dense, and therefore the group
$$
\Gamma_{\OO}^1=\OO^1/\{\pm 1\}
$$
is an irreducible lattice in $\psl_2(\RR)\times\psl_2(\RR)$. An irreducible lattice $\Gamma<\psl_2(\RR)\times\psl_2(\RR)$ is called \textbf{arithmetic} if there exists $k,B$ and $\OO$ as above such that $\Gamma$ is commensurable with $\Gamma_{\OO}^1$. As mentioned before, by Margulis' arithmeticity theorem, every irreducible lattice in $\psl_2(\RR)\times\psl_2(\RR)$ is arithmetic. We will denote by $\cc(k,B)$ the set of all lattices commensurable with a group $\Gamma_{\OO}^1$ up to conjugacy, where $\OO$ is a maximal order in the given quaternion algebra $B$ over a totally real number field $k$.\\

If the lattice $\Gamma$ is cocompact, which is equivalent to saying that the associated quaternion algebra $B$ is a division algebra, the orbit space $X_{\Gamma}:=\Gamma\backslash \HH\times \HH$ is a compact two-dimensional complex analytic space which moreover has a structure of a normal projective surface and which is smooth if $\Gamma$ is torsion-free, that is, if $\Gamma$ has no elements of finite order different from the identity. The numerical invariants of $X_{\Gamma}$ can be computed using a result of Matsushima and Shimura.
\begin{lemma}[Matsushima-Shimura \cite{MatsushimaShimura63}]
Let $\Gamma$ be an irreducible cocompact and torsion-free lattice in $\psl_2(\RR)\times\psl_2(\RR)$ and $X_{\Gamma}=\Gamma\backslash \left(\HH\times \HH\right)$. Let $h^{i,j}(X_\Gamma)=\dim H^j(X_{\Gamma},\Omega^i_{X_{\Gamma}})$ denote the Hodge numbers of $X_{\Gamma}$. Then the following hold
\begin{enumerate}
\item[a)] $h^{i,j}(X_{\Gamma})=0$ for $i\neq j$ and $i+j\neq 2$,
\item[b)] $h^{1,1}(X_{\Gamma})=2(p_g(X_{\Gamma})+1)$ (and $h^{0,0}(X_{\Gamma})=h^{2,2}(X_{\Gamma})=1$).
\end{enumerate}
There, $p_g=h^{2,0}(X_{\Gamma})$ denotes the geometric genus.
\end{lemma}
From this Lemma we imediately get (see \cite{Shavel78}):  
\begin{enumerate}
\item[i)] $q(X_{\Gamma})=0$ 
\item[ii)] $\chi(X_{\Gamma})=1+p_g(X_{\Gamma})$
\item[iii)] $c_2(X_{\Gamma})= 4\chi(X_{\Gamma})$
\item[iv)] $c_1^2(X_{\Gamma})=8\chi(X_{\Gamma})$
\item[v)] $b_1(X_{\Gamma})=b_3(X_{\Gamma})=0$, $b_2(X_{\Gamma})=2(2p_g(X_{\Gamma})+1)$
\item[vi)] $X_{\Gamma}$ is a surface of general type
\end{enumerate}
There, $q$ denotes the irregularity, $\chi$ is the arithmetic genus ($=$ Euler-Poincar\'e characteristic of the structure sheaf), $b_i$ are the Betti numbers, $c_2$ is the Euler number and $c_1^2$ is the selfintersection number of the canonical divisor. From this, we simply read off the following characterization of fake quadrics among all $X_{\Gamma}$.

\begin{lemma}
Let $\Gamma$ be irreducible cocompact and torsion-free lattice in $\psl_2(\RR)\times \psl_2(\RR)$. Then, $X_{\Gamma}$ is a fake quadric if and only if $\chi(X_{\Gamma})=1$. 
\end{lemma}

Let $\nu=\frac{dx_1\wedge dy_1}{y_1^2}\wedge\frac{dx_2\wedge dy_2}{y_2^2}$ denote the hyperbolic volume element on $\HH\times \HH$ and 
$$
vol(X_{\Gamma})=\int_{\mathcal F(\Gamma)} d\nu
$$
the volume of $X_{\Gamma}$, where integral is taken over a fundamental domain $\mathcal F(\Gamma)$ of $\Gamma$. By virtue of the Theorem of Gauss-Bonnet we get
$$
c_2(X_{\Gamma})=\frac{1}{4\pi^2}vol(X_{\Gamma})\ \text{and}\ \chi(X_{\Gamma})=\frac{1}{16\pi^2}vol(X_{\Gamma}) 
$$
in the case of a torsion-free lattice $\Gamma$. But, as the hyperbolic volume is a meaningful notion also in the case of a lattice containing torsions, it makes sense to define for a not necessarily torsion-free lattice $\Gamma<\psl_2(\RR)\times \psl_2(\RR)$
$$
\chi(\Gamma)=\frac{1}{16\pi^2}vol(X_{\Gamma})
$$
Then $\chi(\Gamma)=\chi(X_{\Gamma})$ if $\Gamma$ is torsion-free and for a subgroup $\Gamma<\Delta$ of finite index we have $\chi(\Gamma)=[\Delta:\Gamma]\cdot \chi(\Delta)$.  Since every lattice $\Gamma$ is contained in a maximal lattice $\Delta$, we know that if $X_{\Gamma}$ is a fake quadric we have $\chi(\Delta)\leq 1$ and $1/\chi(\Delta)$ is an integer. The strategy will be first to list all maximal lattices $\Delta$ with this property and then to search for torsion-free subgroups of index $1/\chi(\Delta)$ in $\Delta$.

\subsection{Maximal arithmetic lattices and volumes}
\label{maximalandvol}
As in the previous section, let $k$ be a totally real number field of degree $n$, $\oo_k$ the ring of integers in $k$, $B$ a quaternion algebra over $k$, ramified at $n-2$ infinite places of $k$ and $\OO$ a maximal order in $B$. Recall that the isomorphy class of $B$ is uniquely determined by the so-called reduced discriminant, the formal product 
$$
d_B=\prod_{v\in R_f} \pP_v,
$$
where $R_f=R_f(B)$ is the set of all finite places of $k$ at which $B$ is ramified and $\pP_v$ is the valuation ideal corresponding to a finite place $v\in V_k^f$. Let $r_f=r_f(B)=|R_f|$ be the number of primes at which $B$ is ramified. By the well known theorem of Hasse (see \cite[III, Theoreme 3.1]{VignerasAlgebres}), $r_f+(n-2)$ (which is the total number of ramified places in $B$) is an even number. When working with maximal lattices it is more convinient to interprete $\aut(\HH)\times \aut(\HH)$ as the group $\PGL_2^+(\RR)\times \PGL_2^+(\RR)$, where $\PGL_2^+(\RR)=\GL_2^+(\RR)/\RR^{\ast}$ and $\GL_2^+(\RR)=\{x\in \GL_2(\RR)\mid \det(x)>0\},$ rather then $\psl_2(\RR)\times \psl_2(\RR)$. 
For a totally positive element $x\in k$ we will write $x\succ 0$ and we define $k_+=\{a\in k\mid a\succ 0\}$ as well as
$$
B^+=\{x\in B^{\ast}\mid \Nrd(x)\in k_+\},
$$
where $B^{\ast}=\{x\in B\mid \Nrd(x)\neq 0\}$ denotes the group of all units in $B$. Further we define some groups associated with a maximal order $\OO$:
\begin{align*}
\OO^{\ast}&=B^{\ast}\cap \OO \\
\OO^+ &=B^+\cap \OO\ \text{and}\ \Gamma^+_{\OO}=\OO^+ /\oo_{k}^{\ast} \\
N\OO^+ &=\{x\in B^+\mid x \OO x^{-1}=\OO\}\ \text{and}\ N\Gamma^+_{\OO}=N\OO^+ /k^{\ast}
\end{align*}
Let $\oo_{R_f}=\{a\in \oo_k\mid a\in \oo_{k_v}\ \text{for all}\ v\notin R_f\}$ be the ring of $R_f$-integers and $\oo^{\ast}_{R_f,+}=\oo^{\ast}_{R_f}\cap k_+$. Let
$$
B^+_{R_f}=\{x\in B\mid \Nrd(x)\in \oo_{R_f,+}^{\ast}\}\ \text{and}\ \Gamma_{R_f}^+=B^+_{R_f}/\oo^{\ast}_{R_f}.
$$
We have a chain of normal inclusions:
$$
\Gamma_{\OO}^1\triangleleft \Gamma_{\OO}^+\triangleleft\Gamma^+_{R_f}\triangleleft N\Gamma_{\OO}^+.
$$

\begin{lemma}[\cite{Borel81}, \cite{ChinburgFriedman86}]
\label{indices}
Let $\mathcal I_k$ be the group of all fractional ideals in $k$, $\mathcal P_k$ the group of principal fractional ideals, $\mathcal P_{k,+}=\{(a)\in \mathcal P_k\mid a\succ 0\}$, $\mathcal I_{R_f}$ the subgroup of $\mathcal I_k$ generated by the prime ideals which correspond to places in $R_f(B)$. Let $J_1=\mathcal I_k/\mathcal I_{R_f}\mathcal P_{k,+}$ and $J_2$ the image of $\mathcal P_k$ in $J_1$. Finally, let $J^{(2)}_1<J_1$ be the kernel of the map $\aaa\mapsto \aaa^2$ on $J_1$. Then
\begin{itemize}
\item[a)] $[\Gamma_{R_f}^+:\Gamma_{\OO}^1]=[\oo_{R_f,+}^{\ast}:(\oo_{R_f}^{\ast\ 2})]$
\item[b)] $[N\Gamma_{\OO}^+:\Gamma_{R_f}^+]=[J_1^{(2)}:J_2]$
\end{itemize}
Let $k_B$ denote the maximal abelian extension whose Galois group is elementary 2-abelian and which is unramified at all finite places and where all the primes from $R_f(B)$ are totally decomposed. Then
\begin{itemize}
 \item[c)] $[N\Gamma_{\OO}^+:\Gamma_{\OO}^1]=2^{r_f}[k_B:k]$ 
\end{itemize}
\end{lemma}

\begin{remark}
\label{kstrich}
\begin{itemize}
\item The statements of Lemma \ref{indices} are versions of the original statements in \cite{Borel81}, 8.4-8.6 and \cite{ChinburgFriedman86}, Lemma 2.1, where the authors work in the context of the more general group $\PGL_2(\RR)^a\times \PGL_2(\CC)^b$, whereas we need the group $\PGL_2^+(\RR)^2$. The statements are simple translations according to our situation.      
\item Sometimes it is useful instead of $k_B$ to consider the field $k'_B$ which is an extension of $k$ with the same properties as $k_B$ but where additionally all infinite places of $k$ are unramified. As $k$ is totally real, we have (see \cite{NeukirchANT}, p.366, and also \cite{ChinburgFriedman86})
$$
[k_B:k]^{-1}=[k'_B:k]^{-1}\frac{[\oo_k^{\ast}:\oo_{k,+}^{\ast}]}{2^n}.
$$
There, $[\oo_k^{\ast}:\oo_{k,+}^{\ast}]\geq 2$, since $-1\notin \oo_{k,+}^{\ast}$. Since $k'_B$ is a subfield of the Hilbert class field we know that $[k'_B:k]=2^{\beta}$ divides the ideal class number $h_k$ of $k$.
\item Index $[J_1^{(2)}:J_2]$ is also a divisor of the class number $h_k$. In particular, if $h_k=1$ the groups $N\Gamma_{\OO}^+$ and $\Gamma_{R_f}^+$ coincide.
\end{itemize}
\end{remark}

It is known that the group $N\Gamma_{\OO}^+$ is a maximal lattice, that is, not contained in a bigger lattice. But there are more maximal lattices. By a theorem of Borel there are infinitely many non-conjugate maximal lattices in a commensurability class of a given lattice, see \cite{Borel81}, Section 4.6. Following Borel, maximal lattices can be described as follows (here we give a description used in \cite{ChinburgFriedman00} and \cite{MaclachlanRied03} which is related to Borel's original description as explained in loc.cit.):\\
Let $S\subset V_k^f$ be a finite set (possibly empty) of non-archimedean places which is disjoint from $R_f(B)$. For any place $v\in S$ we denote by $Nv$ the norm of $v$, that is, the cardinality of the residue field $\oo_k/\pP_v$ of the corresponding valuation ideal $\pP_v$. For any $v\in S$ choose a maximal local order $\OO'_{v}$ in $B_v=B\otimes_k k_v\cong M_2(k_v)$ such that the distance $d(\OO_v,\OO'_v)$ between $\OO_v=\OO\otimes_{\oo_k}\oo_{k_v}$ and $\OO'_v$ in the tree of maximal orders is one, that is, $[\OO_v:\OO_v\cap \OO'_v]=Nv$. Let $\OO'$ be a maximal order in $B$ such that $\left(\OO'\right)_v=\OO_v$ for all $v\notin S$ and $\left (\OO'\right )_v=\OO'_v$ for $v\in S$. Let $\mathcal E=\OO\cap \OO'$ and define
$$
N\mathcal E=\{x\in B^{+}\mid x\mathcal E x^{-1}=\mathcal E\}\ \text{and}\ \Gamma_{S,\OO}^+=N\mathcal E/\text{center}.
$$

In particular we have $\Gamma_{\emptyset,\OO}^+=N\Gamma_{\OO}^+$. By \cite{Borel81}, Section 4.4. all maximal lattices will be of the form $\Gamma_{S,\OO}^+$ but not every $\Gamma_{S,\OO}^+$ is itself maximal. The maximality depends on subtile local conditions. We will be particularly interested in the maximality of $\Gamma_{S,\OO}^+$, where $S=\{v\}$ consists of a single prime $v\in V_k^f\setminus R_f$. Here, we have the following criterion for maximality (see \cite{MaclachlanRied03}, Section 11.4): $\Gamma_{S,\OO}^+$ is maximal if and only if there exists a totally positive $c\in k$ such that $ord_v(c)\equiv 1\bmod 2$ and $ord_w(c)\equiv 0\bmod 2$ for all $w\in V_k^f\setminus (S\cup R_f)$, where $ord_{\bullet}$ denotes the normalized exponential valuation corresponding to a finite place. Otherwise $\Gamma_{S,\OO}^+$ is not maximal and is contained in $N\Gamma_{\OO}^+$.\\

\begin{thm}[\cite{Borel81},\cite{Shimizu63},\cite{Vigneras76}]
\label{vol1}
Let $\zeta_k(\cdot)$ denote the Dedekind zeta function of the field $k$ and for a finite place $v\in V_k^f$ let $Nv$ be the norm of $v$ ($=|\oo_k/\pP_v|$). Let $d_k$ be the discriminant of $k$. Then, for a maximal order $\OO$ in $B$ we have
$$
\chi(\Gamma_{\OO}^1)=\frac{d_k^{3/2}\zeta_k(2)}{2^{2n-1}\pi^{2n}}\prod_{v\in R_f} (Nv-1)
$$
\end{thm}

\begin{cor}
\label{vol2}
Under the same conditions as in Theorem \ref{vol1} we have
\begin{itemize}
\item $$ \chi(N\Gamma_{\OO}^+)= \frac{d_k^{3/2}\zeta_k(2)}{2^{2n-1+t'}\pi^{2n} [k_B:k]}\prod_{v\in R_f, Nv\neq 2} \frac{Nv-1}{2}$$
where $t'=\#\{v\in R_f\mid Nv=2\}$.
\item Let $S\subset V_k^f$ be such that $\#S<\infty$ and $S\cap R_f=\emptyset$. Then there exists an integer $0\leq m\leq \#S$ such that
$$
\chi(\Gamma_{S,\OO}^+)= \frac{d_k^{3/2}\zeta_k(2)}{2^{2n-1+t'+m}\pi^{2n} [k_B:k]}\prod_{v\in R_f, Nv\neq 2} \left(\frac{Nv-1}{2}\right)\prod_{v\in S} (Nv+1)
$$
\end{itemize}
\end{cor}

\begin{proof}
The first statement is obtained combining Lemma \ref{indices} and Theorem \ref{vol1}. For two commensurable groups $\Gamma$ and $\Gamma'$ define the generalized index $$ [\Gamma:\Gamma']=\frac{[\Gamma:\Gamma\cap \Gamma']}{[\Gamma':\Gamma\cap \Gamma']}.$$     
Then, by \cite{Borel81}, Section 5.3. there is an integer $0\leq m\leq |S|$ such that
$$
[N\Gamma_{\OO}^+:\Gamma_{S,\OO}^+]=2^{-m}\prod_{v\in S} (Nv+1)
$$
Since $vol(\Gamma')=[\Gamma:\Gamma']vol(\Gamma)$ (and therefore $\chi(\Gamma')=[\Gamma:\Gamma']\chi(\Gamma)$) we get the second statement from the first one.
\end{proof}

\begin{remark}
\label{m}
In general, the number $m$ depends strongly on $S$ and $R_f$, but in the case where $S=\{v\}$ consists of a single prime we know that $m=1$ if and only if $\Gamma_{S,\OO}^+$ is maximal (see \cite{MaclachlanRied03}, Section 11.5).    
\end{remark}

Let us write 

\begin{align}
\label{EBES}
E'_B=&\prod_{v\in R_f, Nv\neq 2} e'_v\ \ \ \text{where $e'_v=\frac{Nv-1}{2}$}, \\
E_{S}=&\prod_{v\in S} \sigma_{v}\ \ \ \text{where $\sigma_v=Nv+1$}, 
\end{align}
for some finite set of places $S\subset V_k^f$. Let further $t=\#\{v\in V_k^f\mid 2\ \text{divides}\ Nv\}$ be the number of places lying over $2$ and finally define 
\begin{align}
\label{gkb}
g(k,B)&=\frac{d_k^{3/2}\zeta_k(2)}{2^{2n+t-1}\pi^{2n} [k_B:k]}\notag\\ 
&=\frac{d_k^{3/2}\zeta_k(2)[\oo_k^{\ast}:\oo_{k,+}^{\ast}]}{2^{3n+t-1}\pi^{2n} [k'_B:k]}
\end{align}

Then we have a formula 
\begin{equation}
\label{volformula}
\chi(\Gamma_{S,\OO}^+)=2^{t-t'} g(k,B) E'_B 2^{-m}E_S\ \  \text{with some $0\leq m\leq |S|$}.
\end{equation}

\begin{remark}
\label{min}
Let $\OO$ and $\OO'$ be maximal orders. In \cite{Borel81} it is additionally proved that $[N\Gamma_{\OO}^+:\Gamma_{S,\OO}^+]\geq 1$ and $[N\Gamma_{\OO}^+:\Gamma_{S,\OO}^+]= 1$ if and only if $\Gamma_{S,\OO}=N\Gamma_{\OO'}^+$. Hence, $\chi$ will achieve its minimum on a group $N\Gamma_{\OO}^+$ and for any other group $\Gamma\in \cc(k,B)$ the value $\chi(\Gamma)$ is a positive integral multiple of $2^{-e}\chi(N\Gamma_{\OO}^+)$, where $e$ is the number of places over $2$ not contained in $R_f$ (see \cite{Borel81}, Section 5.4). Moreover, by \cite{Borel81}, Section 8.2, there are only finitely many conjugacy classes of arithmetic lattices $\Gamma$ such that $\chi(\Gamma)\leq c$ for any bound $c>0$.  
\end{remark}

\subsection{Lower volume bounds}
The goal of this section is to obtain lower bounds for the crucial invariant $\chi$. As explained in Remark \ref{min}, for this purpose one has to study the value 

$$
\chi(N\Gamma_{\OO}^+)=2^{t-t'} g(k,B) E'_B,
$$
which is the minimal possible value for $\chi$ among all lattices in $\cc(k,B)$.\\

Since $E'_B\geq 1$ and $t\geq t'$ we have 

\begin{equation}
\label{absch} 
\chi(N\Gamma_{\OO}^+)\geq g(k,B).
\end{equation}
We know that the uniformizing lattices of fake quadrics $\Gamma$ satisfy $\chi(\Gamma)\leq 1$. Moreover, we know that $\chi(\Gamma)\geq \chi(N\Gamma_{\OO}^+)$ for an arbitrary maximal order $\OO$. Thus, we particularly have a condition $1\geq \chi(N\Gamma_{\OO}^+)$.\\

A handable lower bound for $g(k,B)$ is obtained in \cite{ChinburgFriedman86} (note that $g(k,B)$ defined by the formula (\ref{gkb}) coincides with $g(k,B)$ from \cite{ChinburgFriedman86}). First idea is to use the fact that $[k'_B:k]=2^{\beta}$ is a divisor of the class number and to apply the Theorem of Brauer-Siegel which in the case of a totally real number field $k$ of degree $n$ gives the following relation between the class number $h_k$, regulator $R_k$ and the discriminant $d_k$ which is true for any real $s>1$:

\begin{equation}
\label{brauersiegel} 
h_kR_k\leq 2^{1-n}s(s-1)\left(\Gamma(s/2)\right)^n\left( \frac{d_k}{\pi^n}\right)^{s/2}\zeta_k(s).
\end{equation}
There, $\Gamma(\cdot)$ denotes the gamma function. 
In order to bound $h_k$ from above, Chinburg and Friedman use a lower bound for the regulator due to Zimmert. Instead, we will use use slightly better bounds due to Slavutskii (see \cite{Slavutskii92}) and Friedman (see \cite{Friedman89}). Namely, in \cite{Slavutskii92} we find that for any totally real number field $k$ of degree $n$ we have 
\begin{equation}
\label{slavutskii}
R_k\geq 0.003\cdot \exp(0.75\cdot n)=:\rho_s(n).
\end{equation}
On the other hand, Friedman shows (see \cite{Friedman89}, p.~620) that for any totally real $k$ 
\begin{equation}
\label{friedman}
R_k\geq 0.0062\cdot \exp(0.738\cdot n)=:\rho_f(n).
\end{equation}
Note that for $n\leq 60$ we have $\rho_f(n)>\rho_s(n)$. For $n>60$ $\rho_s(n)$ gives a much better bound for $R_k$. Since we will be interested in small values of $n$, we will use Friedman's estimate more then Slavutskii's. 

The inequalities (\ref{absch}), (\ref{brauersiegel}) and (\ref{friedman}) give

\begin{lemma}[Compare \cite{ChinburgFriedman86}, Lemma 3.2]
\label{chfr2}
$$
g(k,B)>\frac{0.0062\cdot \exp(0.738 n)\cdot \zeta_k(2)\pi^{n(\frac{s-4}{2})} d_k^{(3-s)/2}}{s(s-1)\zeta_k(s)\Gamma(s/2)^n 2^{2n+t-1}}
$$
\end{lemma}
\begin{remark}
\label{praktibilitaet}
Let $\Gamma\in \mathcal C(k,B)$ be an irreducible lattice of fake quadric. Then $g(k,B)<\chi(\Gamma)\leq 1$ and the Lemma \ref{chfr2} produces a function $\Phi(n,s)$ in $n$ and $s$ with $1<s<3$ which is an upper bound for the root discriminant $\delta_k=d_k^{1/n}$ of $k$.
The approach of Prasad and Yeung (see \cite{PrasadYeung07}, Section 6) suggests a comparisition of the upper bound $\delta_k<\Phi(n,s)$ with general lower bounds for the root discriminant due to Odlyzko. Unfortunately, the function $\Phi(n,s)$ (which we do not write explicitely) provides too large upper bounds for the discriminant. Instead, we will use the refined inequalities from \cite{ChinburgFriedman86} to find better bounds for the degree and discriminant of the field $k$.   
\end{remark}
\begin{lemma}
\label{psinb}
Let $k$ be totally real number field of degree $n$ and let $[k'_B:k]$ be the degree of the field $k'_B$ over $k$ (see Remark \ref{kstrich}). Define
\begin{align*}
%\Psi_s(n,[k'_B:k])&=0,0107\cdot \exp\left(0,4427\cdot n-\frac{19.0744}{[k'_B:k]}\right)\\ 
\Psi_f(n,[k'_B:k])&=0.0221\cdot \exp\left(0.4307\cdot n-\frac{19.0744}{[k'_B:k]}\right)
\end{align*}
Then $ g(k,B)> \Psi_f(n,[k'_B:k])$ 
%and $g(k,B)> \Psi_f(n,[k'_B:k])$
\end{lemma}
\begin{proof}
Replacing Zimmert's lower bound for the regulator by the bound of Friedman, we have, as in Proposition 3.1 of \cite{ChinburgFriedman86}, the inequality (which follows from Lemma \ref{chfr2} as the Proposition 3.1. in \cite{ChinburgFriedman86} follows from Lemma 3.2 there)
\begin{align}
\label{chfr3}
%\log g(k,B)&> \log(0.003[\oo_k^{\ast}:\oo_{k,+}^{\ast}])-\log(s(s-1))+n\left(0.75-\log\Gamma(\frac{s}{2})+T(s,y)\right)\\ \notag
%\text{and}&\\
\log g(k,B)&> \log(0.0062[\oo_k^{\ast}:\oo_{k,+}^{\ast}])-\log(s(s-1))+n\left(0.738-\log\Gamma(\frac{s}{2})+T(s,y)\right)
\end{align}
with the function $T(s,y)$ defined in \cite{ChinburgFriedman86} on p.~515 and $s,y\in \RR$, $y>0$. 
Let $k'_B$ be the field defined in Remark \ref{kstrich} (Note that in \cite{ChinburgFriedman86} the notation $h(k,2,B)$ for $[k'_B:k]$ is used). Using the lower bound $T(1.4,0.1)>-0.0464-19.0744/(n\cdot [k'_B:k])$ as computed in \cite{ChinburgFriedman86} and the estimate $[\oo_k^{\ast}:\oo_{k,+}^{\ast}]\geq 2$ we obtain the stated inequality $g(k,B)> \Psi_f(n,[k'_B:k])$ (compare this with Lemma 4.3.~in \cite{ChinburgFriedman86}).
\end{proof}
Since $[k'_B:k]=2^{\beta}$ is a power of two, we can write $\Psi_f(n,[k'_B:k])=\psi_f(n,\beta)$ as a function in $(n,\beta)\in \NN\times \NN\cup\{0\}$. This function is increasing when $n$ is increasing as well as $\beta$ is increasing. Also we observe that for fixed $n\leq 8$ we have $\lim_{\beta\rightarrow \infty}\psi_f(n,\beta)<1$. We compute numerically $\psi_f(54,0)>1.44$. It follows that if a commensurability class $\mathcal C(k,B)$ contains an irreducible lattice of a fake quadric then the degree of $k$ satisfies $n=[k:\QQ]\leq 53$. Also we compute numerically $\psi_f(32,1)>1.5$, hence if $\mathcal C(k,B)$ contains a lattice of a fake quadric and $k$ is totally real of degree $32 \leq n\leq 53$, then $[k'_B:k]=1$, i.~e.~$k'_B=k$. On the other hand, using inequalities $t=\#\{v\in V_k^f\mid 2| Nv\}\leq n=[k:\QQ]$, $[\oo_k^{\ast}:\oo_{k,+}^{\ast}]\geq 2$, and $\zeta_k(2)>\zeta(2\cdot n)$, we see from the definition (\ref{gkb}) of $g(k,B)$ that the condition $g(k,B)\leq 1$ implies
$$
\delta_k\leq \left(\frac{2^{(4n-2)}\pi^{2n}[k'_B:k]}{\zeta(2n)}\right)^{\frac{2}{3n}}=:P(n,[k'_B:k]). 
$$
As seen above, for $32 \leq n\leq 53$, the inequality $1\geq g(k,B)>\psi_f(n,[k'_B:k])$ implies $[k'_B:k]= 1$ and for all those values we can numerically compute $P(n,1)$. On the other hand we have Odlyzko's lower bounds for the minimal root discriminant $m_r(n):=\min\{ \delta_k\mid k\ \text{totally real}, [k:\QQ]=n\}$ obtained in \cite{Odlyzko75} (see also \cite{OdlyzkoHomepage}). We obtain that for all $n>33$ the value $P(n,1)$ is strictly smaller then $m_r(n)$ comparing the values of $P(n,1)$ and $m_r(n)$ given in \cite{OdlyzkoHomepage}. For instance, we have $P(34,1)< 28.43$, whereas $m_r(34)>28.82$ (see \cite{OdlyzkoHomepage}).

\begin{lemma}
\label{obereschranke}
If an irreducible lattice $\Gamma\in \cc(k,B)$ defines a fake quadric, then $k$ is a totally real field of degree $n=[k:\QQ]\leq 33$. Moreover, let $k'_B$ be as in Remark \ref{kstrich}. Then we have
\begin{itemize}
\item $20<n\leq 33 \Longrightarrow [k'_B:k]\leq 2$
\item $14<n\leq 20 \Longrightarrow [k'_B:k]\leq 4$
\item $11<n\leq 14 \Longrightarrow [k'_B:k]\leq 8$
\item $n=11\Longrightarrow [k'_B:k]\leq 16$
\item $n=10 \Longrightarrow [k'_B:k]\leq 32$
\item $n=9 \Longrightarrow [k'_B:k]\leq 256$
\end{itemize}

\end{lemma}

\begin{remark}
The degree bound given in Lemma \ref{obereschranke} is by far not optimal. Using Odlyzko's discriminant bounds under the assuption of GRH, we would have that $[k:\QQ]$ is less then 25. This bound is still too big with regard to a concrete list of all fields which possibly lead to commensurability classes of lattices of fake quadrics. Currently, one is able to produce lists of totally real number fields whose root discriminant is less then 14 (\cite{Voight08}). In contrast, the above bounds imply that the root discriminants of possible fields of definition of lattices of fake quadrics are bigger then 28.        
\end{remark}

\section{Fake quadrics associated with a real quadratic field}
Let $k$ be a real quadratic field and $B$ a quaternion algebra over $k$ which is unramified at the two archimedean places of $k$ (i.~e.~$B$ is totally indefinite). Our purpose is to determine irreducible lattices $\Gamma\in \cc(k,B)$ for which $X_{\Gamma}=\Gamma\backslash \HH\times \HH$ is a fake quadric. We proceed as explained before: We know that if $X_{\Gamma}$ is a fake quadric we have $1=\chi(\Gamma)\geq \chi(N\Gamma_{\OO}^+)$. Hence, observing that $E'_B=\prod_{v\in R_f, Nv\neq 2} e'_v\geq 1$, Corollary \ref{vol2} yields
$$
1\geq \frac{d_k^{3/2}\zeta_k(2)}{2^{3+t'}\pi^{4} [k_B:k]}E'_B\geq \frac{d_k^{3/2}\zeta_k(2)}{2^{3+t'}\pi^{4} [k_B:k]}.
$$ 
Replacing $[k_B:k]^{-1}$ by $\frac{1}{2}[k'_B:k]^{-1}$ (see Remark \ref{kstrich}) and using the inequality $t'\leq 2$, we obtain

$$
1\geq \frac{d_k^{3/2}\zeta_k(2)}{2^{6}\pi^{4} [k'_B:k]}.
$$ 

Moreover, from the definition of $[k'_B:k]$ we see that $[k'_B:k]\leq |Cl_k/Cl_k^2|$, where $Cl_k$ denotes the ideal class group of $k$. On the other hand we know from the genus theory that $|Cl_k/Cl_k^2|\leq2^{a-1}$, where $a$ is the number of prime divisors of the discriminant $d_k$ and we can estimate $2^{a-1}$ by $\sqrt{d_k/3}$ from above.

\begin{lemma}
If $\Gamma\in \cc(k,B)$ leads to a fake quadric, where $k$ is a real quadratic field, then $d_k\leq 1285$. 
\end{lemma}

\begin{proof}
Under the assumption of the Lemma, the above discussion leads to the inequality $d_k< \frac{2^6\pi^4}{\sqrt{3}\zeta_k(2)}$. From the inequality $\zeta_k(2)\geq \zeta(4)=\frac{\pi^4}{90}$ we obtain $d_k\leq 3325$. We consult now the existing tables of class numbers obtained from \cite{NF} of real quadratic fields of discriminant $\leq 3325$ and we see that up to three exceptions $d_k=2920, 2305, 1297$ (but where $[k'_B:k]\leq 8$) we always have $h_k\leq 9$. Hence we can replace once again $[k'_B:k]$ by $8$. This leads to the stated inequality.
\end{proof}

\subsection{Integrality of $1/\chi(N\Gamma_{S,\OO}^+)$}
Let $k$ be a real quadratic field and $\Gamma\in \cc(k,B)$ be ireducible lattice such that $X_{\Gamma}$ is a fake quadric. Assume that $\Gamma$ is contained in a maximal lattice $\Gamma_{S,\OO}^+$, where $S$ is a finite set of places such that $S\cap R_f(B)=\emptyset$. The condition $\chi(\Gamma)=1$ implies that 

$$
1/\chi(\Gamma_{S,\OO}^+)=\frac{2^{3+t'+\alpha}\pi^4}{d_k^{3/2}\zeta_k(2)}E'^{{-1}}_B 2^{m}E^{-1}_S
$$
is an integer, where as before we set $[k_B:k]=2^{\alpha}$, $E'_B=\prod_{v\in R_f,Nv\neq 2} e'_v$ and $E_S=\prod_{v\in S}\sigma_v$. Let us define 
$$
E'_{B,2}=\left\{\begin{array}{ll} 3/2 & \text{if $2$ is inert in $k$}\\ 1 & \text{otherwise}  \end{array}  \right. 
$$
With this definition we have $E'_B=E'_{B,2}\cdot\prod_{v\in R_f, v\nmid 2} e'_v=E'_{B,2}\cdot(\text{integer})$. Since $E_S$ is itself an integer, integrality of $1/\chi(\Gamma_{S,\OO}^+)$ implies that
\begin{equation}
 \label{integer}
\frac{2^{3+t'+\alpha+m}\pi^4}{d_k^{3/2}\zeta_k(2)E'_{B,2}}\in \NN.
\end{equation}
Let $\kappa$ denote the quadratic character of the field $k$. Then we can express the value $\zeta_k(2)$ in terms of generalized Bernoulli numbers more explicitely: Let $B_{2,\kappa}$ be the second generalized Bernoulli number associated with $\kappa$. Then we have $\zeta_k(2)=\frac{\pi^4 B_{2,\kappa}}{6 d_k^{3/2}}$. By definition of $E'_{2,B}$, (\ref{integer}) implies that particulary
\begin{equation}
 \label{integer2}
\frac{6\cdot 2^{3+t'+\alpha+m}}{B_{2,\kappa}}\in \NN .
\end{equation}
For $B_{2,\kappa}$ we have the following explicit formula (see \cite{Shavel78})
\begin{equation}
 \label{Bkappa}
B_{2,\kappa}=\frac{1}{d_k}\sum_{1\leq r\leq d_k-1} r^2 \kappa(r).
\end{equation}

\begin{lemma}
Condition (\ref{integer2}) is satisfied only for
$$
d_k=5,8,12,13,17,21,24,28,29,33,41,60,65,69,77,137,145,161,221,285,353,429,712.
$$ 
\end{lemma}
\begin{proof}
We compute directly (with the aid of a computer) the value $T(d_k)=6\cdot 2^{3+t}/B_{2,\kappa}$ for all $d_k\leq 1285$ using (\ref{Bkappa}) and exclude all $d_k$ for which the denominator of $T(d_k)$ is divisible by a prime $\neq 2$. 
\end{proof}
We collect some relevant invariants of commensurability classes $\cc(k,B)$ which may contain lattices of fake quadrics in the Table \ref{tableinvariants}. Note that the unknowns $\alpha$ and $t'$ appearing in $g(k,B)$ and $\chi(N\Gamma_{\OO}^+)$ depend on $R_f(B)$ which has to be determined.
\begin{table}[h]
\begin{center}
\begin{tabular}{ |p{1cm}|p{1cm}|p{1cm}|p{1cm}|p{2cm}|}
\hline
No. & $d_k$ & $h_k$ & $t$ & $2^{\alpha}g(k,B)$ \\
\hline
I &5 &1 &1 &$1/120$ \\
\hline
II &8 &1 &1 &$1/48$ \\
\hline
III &12 &1 &1 &$1/24$ \\
\hline
IV &13 &1 &1 &$1/24$ \\
\hline
V &17 &1 &2 &$1/24$ \\
\hline
VI &21 &1 &1 &$1/12$ \\
\hline
VII &24 &1 &1 &$1/8$ \\
\hline
VIII &28 &1 &1 &$1/6$ \\
\hline
IX &29 &1 &1 &$1/8$ \\
\hline
X &33 &1 &2 &$1/8$ \\
\hline
XI &41 &1 &2 &$1/6$ \\
\hline
XII &60 &2 &1 &$1/2$ \\
\hline
XIII &65 &2 &2 &$1/3$ \\
\hline
XIV &69 &1 &1 &$1/2$ \\
\hline
XV &77 &1 &1 &$1/2$ \\
\hline
XVI &137 &1 &2 &$1$ \\
\hline
XVII &145 &4 &2 &$4/3$ \\
\hline
XVIII &161 &1 &2 &$4/3$ \\
\hline
XIX &221 &2 &1 &$8/3$ \\
\hline
XX &285 &2 &1 &$4$ \\
\hline
XXI &353 &1 &2 &$4$ \\
\hline
XXII &429 &2 &1 &$8$ \\
\hline
XXIII &712 &2 &1 &$64/3$ \\
\hline
\end{tabular}
 \caption{Invariants related to $\cc(k,B)$}
\label{tableinvariants}
\end{center}
\end{table}
With informations from the Table \ref{tableinvariants}, we can already easy exclude some of the fields with large discriminants
\begin{lemma}
\label{largediscriminant}
There exist no irreducible lattices of fake quadrics defined over the field $k=\QQ(\sqrt{712}), \QQ(\sqrt{429}), \QQ(\sqrt{285})$, $\QQ(\sqrt{221})$ and $\QQ(\sqrt{161})$.
\end{lemma}
\begin{proof}
Consider for instance $k=\QQ(\sqrt{178})$ first. The number $2^{\alpha}=[k_B:k]$ equals $4\cdot [k'_B:k]/([\oo_k^{\ast}:\oo^{\ast}_{k,+}])=4\cdot 2^{\beta}/[\oo_k^{\ast}:\oo^{\ast}_{k,+}]$. The fundamental unit $\epsilon_k$ is totally positive, hence we have $[\oo_k^{\ast}:\oo^{\ast}_{k,+}]=2$ and $\alpha=1+\beta$. On the other hand $2^{\beta}$ divides the class number of $k$ and hence $\beta\leq 1$ and $\alpha\leq 2$. Then $g(k,B)=2^{-\alpha} 64/3 \geq 8/3$. By the general formula (\ref{volformula}) we have $\chi(N\Gamma_{\OO}^+)\geq 2^{t-t'}\frac{8}{3}E_B'$ which is always greater then one. Since by Remark \ref{min} $\chi$ achieves its minimum on $N\Gamma_{\OO}^+$, we can exclude $k=\QQ(\sqrt{178})$. The same argument works for $k=\QQ(\sqrt{429})$ and $k=\QQ(\sqrt{353})$.\\  
As next let us consider $k=\QQ(\sqrt{285})$. Here, the fundamental unit $\epsilon_k$ is also totally positive and we have $\alpha=1+\beta$. Since $2^{\beta}$ divides the class number of $k$ we have $\alpha=1, 2$. It follows that $g(k,B)=1$ or $2$. Again by (\ref{volformula}) we have $\chi(N\Gamma_{\OO}^+)=2^{t-t'}2^{a}E_B'$, with $a=0,1$. Since $2$ is inert in $k$, $t'=0$ and because $t=1$, we have $\chi>1$.\\
Let now $k=\QQ(\sqrt{221})$. Fundamental unit is totally positive, the class number is $2$ and, as above, this implies that $\alpha$ can take only the values $1$ or $2$. Since $2$ is inert in $k$ only $t'=0$ is possible. This leads to $\chi(N\Gamma_{\OO}^+)=2^{a}E_B'/3$ with $a=2$ or $3$ which is greater then 1.\\ 
In the last case $k=\QQ(\sqrt{161})$ the fundamental unit is totally positive, $h_k=1$ and $2$ is split in $k$. It follows that $\alpha=1$ and $g(k,B)=2/3$. Then $\chi(N\Gamma_{\OO}^+)=2^{2-t'} 2E_B'/3$ which is $\leq 1$ only for $t'=2$. But in this case, $\chi(N\Gamma_{\OO}^+)$ is not a reciprocal integer. Therefore, no irreducible lattices of fake quadrics in $N\Gamma_{\OO}^+$ are possible. For a maximal group $\Gamma_{S,\OO}^+\neq N\Gamma_{\OO}^+$ the invariant $\chi(\Gamma_{S,\OO}^+)$ is greater then $\chi(N\Gamma_{\OO}^+)$ and is a positive integral multiple of $2^{-e}\chi(N\Gamma_{\OO}^+)$ by Remark \ref{min}. But as $t'=2$, we have $e=0$ and $\chi(\Gamma_{S,\OO}^+)$ is always greater then one.

%But we also have the condition that $1/\chi(N\Gamma_{\OO}^+)$ is an integer and we see that in the given situation this is not the case. By Remark \ref{min} $\chi(\Gamma_{\OO,S})$ is a positive integral multiple of $2^{-e}\chi(N\Gamma_{\OO}^+)$, where $e$ is the number of places over $2$ not contained in $R_f$. Since $2$ is split in $k$, by definition $e=$  
\end{proof}

For the concrete computation of the invariant $\chi$, the following particularly easy formula of Shavel (see \cite{Shavel78}, Theorem 3.1) for $\chi(\Gamma_{\OO}^1)$ is very useful:\\
Let $\kappa$ be the quadratic Dirichlet character associated with the quadratic field $k$, then
$$
\chi(\Gamma_{\OO}^1)=\frac{B_{2,\kappa}}{48}\prod_{v\in R_f}\left( Nv-1\right).
$$
\subsection{Possible quaternion algebras and maximal lattices}
Our goal is to determine possible quaternion algebras $B$ over a real quadratic field $k$, where $d_k$ is among the discriminants given in the Table \ref{tableinvariants} and compatible maximal lattices $\Gamma_{S,\OO}^+\in \cc(k,B)$ for those fields. This will be done by considering each case separately. We will give a detailed description of the approach in some of the cases and only state the result in the remaining cases which are proved in an analoguous way.\\

\subsubsection{Quaternion algebras and lattices over $\QQ(\sqrt{5})$}
Here we are considering $k=\QQ(\sqrt{5})$ and we have $d_k=5$, $t=1$, $h_k=1$. The fundamental unit $\epsilon_k$ of $k$ is not totally positive, hence $[\oo_k^{\ast}:\oo_{k,+}^{\ast}]=4$. Therefore, $2^{\alpha}=[k_B:k]= 2^{\beta}=[k'_B:k]=1$ and $\alpha=0$. Since $2$ is inert in $k$ there are no places $v\mid 2$ with $Nv=2$ and we have $t'=0$. 
Recall the formula
$$
\chi(\Gamma_{S,\OO}^+)=2^{t-t'}g(k,B)E'_{B}2^{-m}E_S.
$$
We will first treat the case where $S=\emptyset$, that is, $\Gamma_{S,\OO}^+=N\Gamma_{\OO}^+$. After determining all possible $N\Gamma_{\OO}^+$ we will use the fact that $\chi(\Gamma_{S,\OO}^+)$ is a positive integral multiple of $2^{-e}\chi(N\Gamma_{\OO}^+)$ to determine also possible $\Gamma_{S,\OO}^+$ with non-empty $S$.\\

In the given case we have $\chi(N\Gamma_{\OO}^+)=\frac{1}{60}E'_B$. The condition that $\chi(N\Gamma_{\OO}^+)$ is a reciprocal integer implies that $E'_B\mid 60$. In Table \ref{table5} we list the possible values of factors $e'_v=\frac{Nv-1}{2}$, 
%and $\sigma_v=Nv+1$ (which we need later) 
where we write $v_p$ for a place of $k$ lying over a rational prime $p$. In order to avoid additional notational difficulties, we do not list conjugate primes, which contribute the same value, separately, but we keep in mind that in the case of a split rational prime we in fact have to take two places into account.  

\begin{table}[h]
\begin{center}
\begin{tabular}{ c |p{0.5cm}|p{0.5cm}|p{0.5cm}|p{0.5cm}|p{0.5cm}|p{0.5cm}|p{0.5cm}}
\hline
$v_p$ & $v_2$ & $v_3$ & $v_5$ & $v_{11}$ & $v_{31}$ & $v_{41}$ & $v_{61}$\\
\hline
$e'_{v_p}$ &3/2 &4 &2 &5 &15 & 20 & 30\\
\hline
%$\sigma_p$ & 5 &10 &6 &30 &12 & 20&- &- &60 &-\\
%\hline
\end{tabular}
\caption{Possible factors $e'_v$}
\label{table5}
\end{center}
\end{table}
Since $B$ is assumed to be unramified at the two archimedean places, and additionally $B$ is a division algebra, there must be at least one ramified finite place. Because $|R_f(B)|$ has to be even, there are at least two places in $R_f(B)$.\\

Assume $v_2\in R_f$. Then, there is $v_p\in R_f$ such that $e'_{v_2}\cdot e'_{v_p}=\frac{3}{2}\cdot e'_{v_p}|60$. Therefore $e'_{v_p}|40$ and consulting the Table \ref{table5} we see that $v_p\in\{v_{41},v_{11},v_{5},v_{3}\}$ with
$$
e'_{v_2}\cdot e'_{v_p}=\left\{\begin{array}{ll} 30 & \text{if $p=41$}\\ 15/2 & \text{if $p=11$}\\ 3 & \text{if $p=5$}\\6 & \text{if $p=3$}\\ \end{array}   \right.
$$
Note, that for $p=41$ and $p=11$ we in fact have two possible conjugate $v_{p}$'s. Now, if $R_f$ contains further places, then at least two, $w_1$ and $w_2$, say. These have to satisfy the relation $e'_{w_1}\cdot e'_{w_2}\mid \frac{60}{e'_{v_2}\cdot e'_{v_p}}$. From the Table \ref{table5} we obtain the only possibilities $\{w_1,w_2\}=\{v_3,v_5\},\{v_3,v_{11}\},\{v_5,v_{11}\}$ and from this the following reduced discriminants (which determine the isomorphy class of $B$):
$$
d_B=v_2v_{41},v_2v_{11},v_2v_5,v_2v_3,v_2v_3v_5v_{11}. 
$$

Assume now that $v_3\in R_f$ but $v_2\notin R_f$. Then there is a further place $v\in R_f$ which satisfies $e'_{v_3}\cdot e'_{v}=4e'_{v}\mid 60\Rightarrow e'_{v}\mid 15$. The only prime which satisfies this relation is $v_{31}$ and we get the additional reduced discriminant $d_B=v_3v_{31}$.\\

We consider now the case in which $v_5\in R_f$ but $v_2,v_3\notin R_f$. Again there is a further place $w\in R_f$ satisying $e'_v\mid 30$ and we have $v\in \{v_{11},v_{31}, v_{61}\}$ with
$$
e'_{v_5}\cdot e'_{w}=\left\{\begin{array}{ll} 10 & \text{if $w=v_{11}$}\\ 30 & \text{if $w=v_{31}$}\\ 60 & \text{if $w=v_{61}$}\\ \end{array}   \right.
$$
No additional primes are possible and we get new discriminants
$$
d_B=v_5v_{11},v_5v_{31},v_5v_{61}.
$$

Assuming $v_{31}\in R_f$ but $v_2,v_3,v_5, v_{11}\notin R_f$ we would have that any other place $w\in R_f$ has to satisfy $e'_w\mid 4$ which is not possible with remaining $w=v_{41}$ and $w=v_{61}$. In the same way the remaining combinations do not lead to new quaternion algebras. Thus we proved
\begin{lemma}
\label{d5a}
Let $k=\QQ(\sqrt{5})$ and $B$ a division quaternion algebra over $k$ unramified at the two archimedean places. If $\cc(k,B)$ contains a lattice of a fake quadric, then the reduced discriminant $d_B$ of $B$ is one of the following (modulo Galois-conjugation)
$$
d_B=v_2v_{41},v_2v_{11},v_2v_5,v_2v_3,v_2v_3v_5v_{11},v_3v_{31},v_5v_{11},v_5v_{31},v_5v_{61}.
$$
\end{lemma}
Now we can also list possible maximal lattices:
\begin{cor}
\label{d5b}
Let $k=\QQ(\sqrt{5})$ and $B$ a division quaternion algebra over $k$ unramified at $V_k^{\infty}$. Assume that $\cc(k,B)$ contains a lattice $\Gamma$ of a fake quadric. Then $\Gamma<N\Gamma_{\OO}^+$ up to the exception $d_B=v_2v_5$ where $\Gamma$ may also be contained in $\Gamma_{S,\OO}^+$ for $S=\{v_3\}$ or $S=\{v_{19}\}$.    
\end{cor}
\begin{proof}
Assume that $S\neq \emptyset$. Under the general assumption, we have $1\geq \chi(\Gamma_{S,\OO}^+)=\chi(N\Gamma_{\OO}^+)2^{-m}E_S$. Since the rational primes $2$ and $3$ are inert in $k$, for all primes $v\in V_k^f$ we have $\sigma_v=Nv+1\geq 5$. Therefore $1\geq \chi(\Gamma_{S,\OO}^+)\geq \chi(N\Gamma_{\OO}^+) 2^{-m}5^{|S|}$. Since $m\leq |S|$, we get $2^{|S|}/\chi(N\Gamma_{\OO}^+)\geq 5^{|S|}$. By Lemma \ref{d5a} we can compute the exact value of $\chi(N\Gamma_{\OO}^+)$ for all possible candidates and this gives an upper bound for $|S|$ which is $3$ for $d_B=v_2v_5$, $2$ for $d_B=v_2v_{11}, v_2v_3$, $1$ for $d_B=v_5v_{11}$ and $0$ in all other cases. On the other hand, we have the condition that $\frac{1}{\chi(\Gamma_{S,\OO}^+)}=\frac{1}{\chi(N\Gamma_{\OO}^+)}2^mE_S^{-1}$ is an integer, and therefore we have that $E_S\mid 2^m\chi(N\Gamma_{\OO}^+)^{-1}$ hence $E_S\mid 2^{|S|}\chi(N\Gamma_{\OO}^+)^{-1}$. We compare now this condition with possible values for $\sigma_v=Nv+1$, which leaves us with the stated possibilities. There, we also use the fact that $R_f\cap S=\emptyset$. By the maximality criterion in Section \ref{maximalandvol}, we see that $\Gamma_{S,\OO}^+$ is maximal. Namely $c=3$ satisfies the condition of that criterion in the first case $d_B=v_2v_5$ and $S=\{v_3\}$ as well as $c=9/2+\sqrt{5}/2$ in the second case $d_B=v_2v_5$ and $S=\{v_{19}\}$. 
\end{proof}

\subsubsection{Quaternion algebras and lattices over $\QQ(\sqrt{2})$}
Let us now consider the case $k=\QQ(\sqrt{2})$. Here, $d_k=8$. Therefore $t=1$ and $t'$ may be $1$ or $0$, depending on whether $v_2$ is in $R_f$ or not. Again, the fundamental unit is not totally positive, hence $[k_B:k]=[k'_B:k]=1$, since $h_k=1$. We obtain
$$
\chi(N\Gamma_{\OO}^+)=\left\{\begin{array}{ll} \frac{1}{48}\cdot E'_B & \text{for $t'=1$}\\ \frac{1}{24}\cdot E'_B & \text{for $t'=0$} \end{array}   \right.
$$
As before, we have a divisibility condition: $E'_B\mid 48$ for $t'=1$ and $E'_B\mid 24$ for $t'=0$. We summarize all possible values $e'_v$ which satisfy this condition in the following table

\begin{table}[h]
\label{table8}
\begin{center}
\begin{tabular}{ c |p{0.5cm}|p{0.5cm}|p{0.5cm}|p{0.5cm}|p{0.5cm}|p{0.5cm}}
\hline
$v_p$ & $v_2$ & $v_3$ & $v_5$ & $v_{7}$ & $v_{17}$ & $v_{97}$\\
\hline
$e'_{v_p}$ &1/2 &4 &12 &3 &8 & 48 \\
\hline
%$\sigma_p$ & 5 &10 &6 &30 &12 & 20&- &- &60 &-\\
%\hline
\end{tabular}
\caption{Possible factors $e'_v$}
\end{center}
\end{table}

If we assume that $t'=1$ we have $v_2\in R_f$. Then, there is a place $v\in R_f$ and $e'_{v}\mid 48$. The above table gives $v\in \{v_3,v_5,v_7, v_{97}\}$ as possibilities (where $v_7$, $v_{17}$ and $v_{97}$ come together with their conjugates). If there is a further place in $R_f$ then at least two of them, $w_1,w_2$, say. Looking at the possible products $e'_{v}\cdot e'_{w_1}\cdot e'_{w_2}$, we see easily that this never gives a divisor of $48$. Therefore only possible discriminants in case $t'=1$ are $d_B=v_2v_3,v_2v_5, v_2v_7, v_2v_{17}, v_2v_{97}$.\\

In the same way, the case $t'=0$ is treated and we get only two further reduced discriminants $d_B=v_3v_7, v_7v_{17}$.

\begin{lemma}
\label{d8ab}
Let $k=\QQ(\sqrt{2})$ and $B$ a division quaternion algebra over $k$ unramified at $V_k^{\infty}$. Assume that $\cc(k,B)$ contains a lattice $\Gamma$ of a fake quadric. Then $$d_B=v_2v_3,v_2v_5, v_2v_7, v_2v_{17}, v_2v_{97},v_3v_7, v_7v_{17}.$$ Moreover, $\Gamma$ is contained in $N\Gamma_{\OO}^+$.
%\begin{center}
%\begin{tabular}{c|c}
%$d_B$ & $S$ \\
%\hline
%$v_2v_3$ &$\{v_7\}$,$\{v_{23}\}$,$\{v_{31}\}$,$\{v_{47}\}$ \\
%\hline
%$v_2v_5$ &$\{v_7\}$\\
%\hline
%$v_2v_7$ &$\{v_{31}\}$,$\{v_{63}\}$,$\{v_{127}\}$\\
%\hline
%$v_2v_{17}$ &$\{v_5\}$,$\{v_7\}$,$\{v_{23}\}$\\

%\hline
%\end{tabular}
%\end{center}

%where $\Gamma$ may also be contained in $\Gamma_{S,\OO}^+$ for $d_B$ and $S$ given in the table.    
\end{lemma}

\subsubsection{Quaternion algebras and lattices over $\QQ(\sqrt{3})$}
For $d_k=12$ we have $h_k=1$, $t=1$ and $t'=0\ \text{or}\ 1$. Since the fundamental unit of $k$ is totally positive we have $[\oo_k^{\ast}:\oo_{k,+}^{\ast}]=2$ and therefore $\alpha= 1$, since $2^{\alpha}=2^{1+\beta}$ and $2^{\beta}=[k_B':k]$ divides the class number $h_k$. We get $\chi(N\Gamma_{\OO}^+)= 1/48$ for $t'=1$ and $\chi(N\Gamma_{\OO}^+)=1/24$ for $t'=0$ and therefore the condition $E'_B\mid 48$ or $E'_B\mid 24$ according to the two cases $t'=1$ \text{and} $t'=0$.

\begin{lemma}
\label{d12ab}
Let $k=\QQ(\sqrt{3})$ and $B$ a division quaternion algebra over $k$ unramified at $V_k^{\infty}$. Assume that $\cc(k,B)$ contains a lattice $\Gamma$ of a fake quadric. Then $$d_B=v_2v_3,v_2v_5, v_2v_7,v_2v_{13},v_3v_5, v_3v_{13}.$$ Moreover, $\Gamma<N\Gamma_{\OO}^+$ up to the exceptions: \\
\begin{table}[h]
\begin{center}
\begin{tabular}{c|c}
\hline
$d_B$ & $S$ \\
\hline
$v_2v_3$ &$\{v_{11}\}$,$\{v_{23}\}$ \\
\hline
$v_2v_{13}$ &$\{v_3\}$\\
\hline
\end{tabular}
\caption{possible non-empty $S$}
\label{nonemptyS12}
\end{center}
\end{table}

where $\Gamma$ may also be contained in $\Gamma_{S,\OO}^+$ for $d_B$ and $S$ given in the Table \ref{nonemptyS12}.    
\end{lemma}
\begin{proof}
Formal computations as in the previous cases with the values from the Table \ref{table12} lead to the following possible reduced discriminants: 
$$d_B=v_2v_3,v_2v_5, v_2v_7, v_2v_{13}, v_2v_{97},v_3v_5, v_3v_7, v_3v_{13}.$$

Having these concrete possible discriminants, we can use Lemma \ref{indices} and can compute the index $[N\Gamma_{\OO}^+:\Gamma_{\OO}^1]=[\oo_{R_f,+}^{\ast}:(\oo_{R_f}^{\ast})^2]$ more precisely. Having this information, we see that for $d_{B}=v_2v_{97}$ and $d_B=v_3v_7$, $\chi(N\Gamma_{\OO}^+)>1$.\\
Assume that a group $\Gamma_{S,\OO}^+$ contains a lattice of a fake quadric. Again, formal computation as in the proof of Corollary \ref{d5b} shows that the only possible $S$ are as given in the statement. Now, one uses the criterion described in Section \ref{maximalandvol} to show that $\Gamma_{S,\OO}^+$ is indeed maximal. Consider for instance the case $d_B=v_2v_3$ and $S=\{v_{11}\}$. Note that the prime ideal $\pP_{v_{11}}$ is generated by $\pi_{11}=1+2\sqrt{3}$ which is not totally positive. Also, generators $\pi_2=1+\sqrt{3}$ and $\pi_3=\sqrt{3}$ of ramified primes $v_2$ and $v_3$ are not totally positive but $\pi_{11}\cdot\pi_2$ and $\pi_{11}\cdot\pi_2$ are. Hence, $c=\pi_{11}\pi_{2}$ is in $k_+$ with $ord_{v_{11}}(c)$ is odd and $ord_w(c)$ is even for all $w\in V_k^f\setminus (S\cup R_f)$. The other cases are treated in the similar manner.  
\end{proof}
\begin{table}[h]
\begin{center}
\begin{tabular}{ c |p{0.5cm}|p{0.5cm}|p{0.5cm}|p{0.5cm}|p{0.5cm}|p{0.5cm}}
\hline
$v_p$ & $v_2$ & $v_3$ & $v_5$ & $v_{7}$ & $v_{13}$ & $v_{97}$\\
\hline
$e'_{v_p}$ &1/2 &1 &12 &24 &6 & 48 \\
\hline
%$\sigma_p$ & 5 &10 &6 &30 &12 & 20&- &- &60 &-\\
%\hline
\end{tabular}
\caption{Possible factors $e'_v$}
\label{table12}
\end{center}
\end{table}

\subsubsection{Remaining cases with $h_k=1$}
All other cases where $k$ is a real quadratic field with class number one are treated in the same way as the cases above. 

\begin{lemma}
\label{classno1ab}
The following table contains all possible commensurability classes $\cc (k,B)$, where $k$ is a real quadratic field (with $d_k\neq 5,8,12$) of class number one and maximal lattices inside the commensurability class which contain lattices of a fake quadric.  

\begin{center}
\begin{tabular}{|c|c|c|c|c|c|c|}
\hline
$d_k$ & 13 & 13& 13& 13& 13& 13 \\
\hline
$d_B$ & $v_2v_3$ & $v_2v_{17} $ & $v_2v_3v'_3 v_{17}$ & $v_3v'_3$ & $v_3 v_5$ & $v_3 v_{13}$ \\
\hline
$S$&-- & $\{v_3\}$ & -- & -- &-- &--\\
\hline
\hline
$d_k$ & 17 & 17& 17& 17& 17& 21\\
\hline
$d_B$ & $v_2v'_2$ & $v_2v'_{2}v_3v_{13}$ & $v_2v_3$ & $v_2 v_5$ & $v_2v_{13}$ & $v_2v_3$\\
\hline
$S$ & $\{v_{47}\}$ &-- &-- &-- &-- &$\{v_7\}$\\
\hline
\hline
$d_k$ & 21 & 21& 21& 21& 21& 24\\
\hline
$d_B$ & $v_2 v_{5}$ & $v_3v_5$ & $v_2 v_{7}$ & $v_3v_{7}$ & $v_5v_7$ & $v_2v_3$\\
\hline
$S$  & $\{v_3\}$ &-- &-- &-- &--& --\\
\hline
\hline
$d_k$ &24 & 24& 28 & 28 & 29& 33 \\
\hline
$d_B$ & $v_2 v_{5}$ & $v_3v_5$ & $v_2v_3$ & $v_2v_7$& $\emptyset$ & $v_2v'_2$ \\
\hline
$S$ & $\{v_3\}$ &-- &-- & -- & --&$\{v_3\}$ \\
\hline

\hline
$d_k$ & 33 &33 & 41 & 69  & 77& 137 \\
\hline
$d_B$ & $v_2v_3$& $v_2v_{17}$ & $v_2v'_2$ &$\emptyset$ & $\emptyset$& $v_2v'_2$ \\
\hline
$S$  & --&-- &$\{v_3\}$ & --&--& -- \\
\hline
\hline

%\hline
\end{tabular}
\end{center}
There, the entry in the row "$S$" indicates the possible non-empty finite sets $S$ of finite places such that the maximal lattice $\Gamma_{S,\OO}^+$ may contain an irreducible lattice of a fake quadric. Otherwise such a lattice, if it exists, is conatined in a maximal lattice $N\Gamma_{\OO}^+$. Since the only everywhere unramified quaternion algebra over $k$ is the matrix algebra $M_2(k)$ and the commensurability class $\mathcal C(k, M_2(k))$ only contains non-compact lattices, there are no irreducible lattices of fake quadrics over the real quadratic fields $\QQ(\sqrt{29})$, $\QQ(\sqrt{69})$ and $\QQ(\sqrt{77})$.
\end{lemma} 

\subsubsection{Remaining cases with $h_k>1$}

We will ilustrate our approach on the example $k=\QQ(\sqrt{15})$. Here $t=1$ and $t'=1\ \text{or}\ 0$, $h_k=2$. Therefore $\beta$ may be $0$ or $1$, that is, $k'_B$ is either the field $k$ itself or the Hilbert class field $Hilb(k)$ of $k$. The fundamental unit is totally positive, hence $\alpha= \beta+1$. We obtain
$$
\chi(N\Gamma_{\OO}^+)= \left\{\begin{array}{cl} \frac{1}{2}E'_B & \text{for}\ t'=0, \beta=0\\
					       \frac{1}{4}E'_B & \text{for}\ t'=1, \beta=0\\ 
					       \frac{1}{4}E'_B & \text{for}\ t'=0, \beta=1\\
					       \frac{1}{8}E'_B & \text{for}\ t'=1, \beta=1\\
					       \end{array}   \right.
$$
From this we get a divisibility conditions: $E'_B\mid 4$ if $\beta=0$, and $E'_B\mid 8$ if $\beta=1$ which provide the possible places of ramification: $v=v_2$ with $e'_{v_2}=1/2$, $v=v_3$ with $e'_{v_3}=1$, $v=v_5$ with $e'_{v_5}=2$ and $v=v_{17}$ with $e'_{v_{17}}=8$. By definition of $k_B$ and $k'_B$, all the primes in $R_f(B)$ split completely in $k_B$ and $k'_B$. But the only prime ideals which split completely in $Hilb(k)$ are the principal ideals. But, one can check that the prime ideals corresponding to $v_2$, $v_3$ and $v_5$ are not principal. Namely, if the prime ideal $\pP_{v}$ is principal, then there is an element $\pi=a+b\sqrt{15}\in \oo_k$ such that $Nv=\pm N_{k/\QQ}(\pi)$. But, as the integral binary quadratic form $x^2-15y^2$ does not represent $\pm 2, \pm 3$, $\pm 5$ and $\pm 17$, there are no elements in $\oo_k$ with norm $\pm p$ for $p=2,3,5,17$. Therefore $\beta=1$ is not possible. Analyzing the remaining cases we get
\begin{lemma}
\label{d60ab}
Let $k=\QQ(\sqrt{15})$ and $B$ a division quaternion algebra over $k$ unramified at $V_k^{\infty}$. Assume that $\cc(k,B)$ contains a lattice $\Gamma$ of a fake quadric. Then $$d_B=v_2v_3,v_2v_5, v_3v_5.$$ Moreover, $\Gamma<N\Gamma_{\OO}^+$ up to the exception $d_B=v_2v_5$ where $\Gamma$ may also be contained in $\Gamma_{S,\OO}^+$ for $S=\{v_3\}$.
\end{lemma}

Similar arguments apply to the other cases and we get (see also Lemma \ref{largediscriminant}):

\begin{lemma}
\label{classno>1ab}
There are no irreducible lattices of fake quadrics defined over $k=\QQ(\sqrt{145})$. If $\Gamma$ is an irreducible lattice of a fake quadric and $\Gamma \in \cc(\QQ(\sqrt{65}), B)$ then $d_B=v_2v'_2$, where $v_2$ and $v'_2$ are the two conjugate primes which lie over 2. Moreover, $\Gamma$ is contained in $N\Gamma_{\OO}^+$.  
\end{lemma}

\subsection{Elements of finite order}
Establishing possible maximal irreducible lattices in $\cc(\QQ(\sqrt{d}),B)$, which may contain lattices of fake quadrics, we now have to find torsion-free subgroups of index $1/\chi(\Gamma_{S,\OO}^+)$ in the given maximal lattice $\Gamma_{S,\OO}^+$. In \cite{ChinburgFriedman00}, Chinburg and Friedman give precise necessary and sufficient conditions for the existence of finite order elements in maximal arithmetic lattices arising from quaternion algebras. With the knowledge of all the orders of torsion elements in $\Gamma_{S,\OO}^+$ we can check a necessary condition for the existence of torsion-free subgroups of given index. 
\begin{eqnarray}
\label{vielfache} 
\textit{Assume that $\Gamma_{S,\OO}^+$ contains a finite subgroup $U$ and let $\Gamma$ a torsion-free}\\ 
\textit{ subgroup of $\Gamma_{S,\OO}^+$. Then the index $[\Gamma_{S,\OO}^+:\Gamma]$ is an integral multiple of $|U|$.}\notag
\end{eqnarray}

Precise criteria for existence of elements of finite order in maximal arithmetic subgroups in unit groups of quaternion algebras are given in \cite{ChinburgFriedman00}. Since we are working with arithmetic subgroups of $B^+/k^{\ast}$ we need a slight adaption of the conditions given there. Such an adaption has already been used in \cite{Maclachlan06} where the author considers maximal arithmetic lattices in $\psl_2(\RR)$. In order to formulate the relevant results, let us first introduce some notations. Let $k$ be totally real and $B$ be a division quaternion algebra over $k$ which is unramified at least over one infinite place (Eichler condition). For a finite set of prime ideals $S$ let $\mathcal I(S)$ be the subgroup of the ideal group $\mathcal I_k$ generated by the ideals in $S$. Define 
$$
H(S,B)=\{a\in k_+^{\ast}\mid a\oo_k\in \mathcal I(S)\mathcal I(R_f)\mathcal I_k^{2}\}
$$
If $S$ is empty let us write $H(B)$  for $H(\emptyset, B)$.
\begin{lemma}[\cite{ChinburgFriedman00},\cite{Maclachlan06}]
\label{torsioncrit}
Let $k$ and $B$ be as above and let $\zeta_m$ be a primitive $m$-th root of unity. Then a maximal lattice $\Gamma_{S,\OO}^+$ contains an element of order $m>2$ if and only if all of the following statements a)--d) are true:
\begin{enumerate}
\item[a)] $\zeta_m+\zeta_m^{-1}\in k$
\item[b)] Every prime $v\in R_f(B)$ is non-split in $k(\zeta_m)$
\item[c)] The ideal generated by $(1+\zeta_m)(1+\zeta_m^{-1})$ lies in $\mathcal I(S)\mathcal I(R_f)\mathcal I_k^2$ 
\item[d)] For each $v\in S$ at least one of the following conditions holds
\begin{itemize}
\item $v$ is split in $k(\zeta_m)$
\item The normalized $v$-valuation $ord_{v}\left( (1+\zeta_m)(1+\zeta_m^{-1})\right)$ is odd
\item $(\zeta_m-1)(\zeta_m^{-1}-1)\in \pP_v$.
\end{itemize}
\end{enumerate}  
$\Gamma_{S,\OO}^+$ contains an element of order $2$ if and only if both statements $\alpha)$ and $\beta)$ are true
\begin{enumerate}
\item[$\alpha)$] There exists $[a]\in H(S,B)/k^{\ast^{2}}$ such that every prime $v\in R_f$ is non-split in $k(\sqrt{-a})$
\item[$\beta)$]  For every $v$ in $S$ at least one of the following conditions holds
\begin{itemize}
\item $ord_v(a)$ is odd.
\item $v$ splits in $k(\sqrt{-a})$
\item $v\mid 4\oo_k$.
\end{itemize}
\end{enumerate}  
\end{lemma}
\begin{remark}
\label{remtorsioncrit}
\begin{enumerate}
\item Since we are working over a quadratic field, the condition a) in Lemma \ref{torsioncrit} implies that $\varphi(m)\mid 4$ (with $\varphi(\cdot)$ denoting the Euler's phi function) and restricts the possible orders of torsion elements to $m=2,3,4,5,6,8,10,12$. Also, Lemma \ref{torsioncrit} implies that for odd $m$, an element of order $m$ exists in $\Gamma_{S,\OO}^+$ if and anly if an element of order $2m$ is in $\Gamma_{S,\OO}^+$.  
\item Note that elements of order $m>2$ in $N\Gamma_{\OO}^{+}$ are already contained in $\Gamma_{\OO}^1$ since they come from a root of unity in $B$ which necesseraly is of reduced norm one. An element of order two lies in $\Gamma_{\OO}^1$ if and only if $k(\sqrt{-a})=k(\sqrt{-1})$ with the notations as in the Lemma \ref{torsioncrit} for the same reason.
\item Let $m>2$ be a prime, then the condition on splitting in $k(\zeta_m)$ of primes $v$ not dividing $m$ is particulary easy: $v$ splits in $k(\zeta_m)$ if and only if $Nv\equiv 1\bmod m$. There is also a similar simple criterion for the existence of elements of order $4$, see \cite{ChinburgFriedman00}, Corollaries 3.4 and 3.5. 
\end{enumerate}
 
\end{remark}

\subsubsection{Lattices of fake quadrics over $\QQ(\sqrt{5})$}
Here we have $d_k=5$ and possible reduced discriminants and maximal lattices are given in Lemma \ref{d5a} and Corollary \ref{d5b}. We first observe that the raduced discriminants $d_B=v_2v_3v_5v_{11}$, $v_3v_{31}$ and $v_5v_{61}$ do not lead to fake quadrics. Namely (compare \cite{Shavel78}, Lemma 6.1), in those cases the only candidate for the lattice of a fake quadric is the group $N\Gamma_{\OO}^+$ itself, which always contains elements of order $2$. This follows from Lemma \ref{torsioncrit}, condition $\alpha)$ by taking the element $a=2\cdot 3\cdot(5+\sqrt{5})(\frac{7+\sqrt{5}}{2})$ for $d_B=v_2v_3v_5v_{11}$, $a=3\cdot (6+\sqrt{5})$ for $d_B=v_2v_{31}$ and $a=(5+\sqrt{5})(31+30\sqrt{5})$ for $d_B=v_5v_{61}$ which in all the cases is in the corresponding group $H(B)$. One easily sees that $[a]\in H(B)/k^{\ast^2}$ satisfies condition $\alpha)$ in Lemma \ref{torsioncrit} since every prime in $R_f$ is ramified in $k(\sqrt{-a})$. As next, let us show that all the other commensurability classes contain lattices of fake quadrics.\\

\noindent Consider $d_B=v_2v_{41}$. First note that, since $41$ is split in $k$, we have two choices of the place $v_{41}$, namely, either the place represented by $\pi_{41}=\frac{13+\sqrt{5}}{2}$ or by $\pi'_{41}=\frac{13-\sqrt{5}}{2}$. In any case, $\chi(N\Gamma_{\OO}^+)=1/2$ and we need to give a torsion-free lattice of index two in $\Gamma$. Note that $N\Gamma_{\OO}^+/\Gamma_{\OO}^1$ is also isomorphic to $H(B)/k^{\ast^{2}}$ just by interpreting $H(B)$ as the group generated by reduced norms of elements in $N\Gamma_{\OO}^+$. Let us concretely assume, that $v_{41}$ is represented by $\pi_{41}$, then $H(B)/k^{\ast^{2}}=\{[1], [\pi_2], [\pi_{41}], [\pi_2\pi_{41}]\}$ where $\pi_2=2$ and $\pi_{41}=\frac{13+\sqrt{5}}{2}$ are the totally positive generators of the prime ideals $\pP_{v_{p}}$, $v_p\in R_f$. There are no elements of order $m>2$ in $N\Gamma_{\OO}^+$ by Lemma \ref{torsioncrit}, b), because $\pi_2$ is split in $k(\zeta_3)$ and $\pi_{41}$ is split in $k(\zeta_4)$ as well as in $k(\zeta_5)$. In order to determine the elements of order $2$, we need to check if primes $\pi_2$ and $\pi_{41}$ are non-split in $k(\sqrt{-a})$, $[a]\in H(B)/k^{\ast^2}$. As in the remark above, $a=\pi_2\pi_{41}$ will lead to an element of order $2$, because $\pi_2$ and $\pi_{41}$ are ramified in $k(\sqrt{-a})$ (Lemma \ref{torsioncrit}, $\alpha)$). On the other hand $a=\pi_{2}$ does not lead to a torsion because $\pi_{41}$ is split in $k(\sqrt{-a})=k(\sqrt{-2})$. In the same way, $\pi_{41}$ does not lead to a torsion because $2$ is split in $k(\sqrt{-\pi_{41}})$. We did the corresponding calculations with the help of the computer algebra system PARI. Let $\Gamma_{[\pi_2]}$ be the subgroup of $N\Gamma_{\OO}^+$ defined as the inverse image of $\langle [\pi_2]\rangle<H(B)/k^{\ast^2}$ under the homomorphism $N\Gamma_{\OO}^+\longrightarrow H(B)/k^{\ast^2}$ (instead we could also take $\langle [\pi_{41}]\rangle$ and look at the group $\Gamma_{[\pi_{41}]}$ defined in the same way as $\Gamma_{[\pi_2]}$). Then, $\Gamma_{[\pi_2]}$ (and $\Gamma_{[\pi_{41}]}$) is a torsion-free subgroup in $N\Gamma_{\OO}^+$ of index $2$ and $X_{\Gamma_{[\pi_2]}}$ (as well as $X_{\Gamma_{[\pi_{41}]}}$) is a fake quadric. Note that we have a similar picture, when considering the place $v'_{41}$ beeing represented by $\pi'_{41}$. Also note that all these fake quadrics as well as their construction already appear in \cite{Shavel78}.\\

\noindent We find a similar situation when looking at $d_{B}=v_5v_{31}$ and $d_B=v_5v'_{31}$. The appropriate index two subgroups $\Gamma=\Gamma_{[\pi_{\bullet}]}$ will lead to fake quadrics as already described in \cite{Shavel78}.\\

\noindent Let us now assume that $d_B=v_2v_{11}$ (the same will be the case with $d_B=v_2v'_{11}$). Here $\chi(N\Gamma_{\OO}^+)=1/8$ and $\chi(\Gamma_{\OO}^1)=1/2$. The group $\Gamma_{\OO}^1$ contains only elements of order $2$ by Lemma \ref{torsioncrit} (there are no elements of order $3$ and $5$ by Remark \ref{remtorsioncrit}). Let $\Pi_{11}$ be the unique ideal in $\OO$ such that $\Pi_{11}^2=\pP_{11}\OO$. Consider the congruence subgroup 
$$ \Gamma_{\OO}^1(\Pi_{11})=\{x\in \Gamma_{\OO}^1\mid x\equiv 1\bmod \Pi_{11}\}.$$
It is a normal subgroup of $\Gamma_{\OO}^1$ and the index is computed using a result of C.~Riehm (see \cite{Riehm70} Theorem 7, also \cite{Shavel78}, Section 5, or \cite{VignerasAlgebres}, p.~108-109):
The quotient $\Gamma_{\OO}^1/\Gamma_{\OO}^1(\Pi_{11})$ is isomorphic to the quotient $B^1_{v_{11}}/\pm B^1_{v_{11}}(\bar \Pi_{11})$, where $B^1_{v_{11}}$ is the norm-1 group of the local algebra $B\otimes_k k_{v_{11}}$ and $\bar\Pi_{11}$ is the topological closure of $\Pi_{11}$ in this localization. Since $B^1_{v_{11}}/B^1_{v_{11}}(\bar\Pi_{11})$ is isomorphic to the kernel $\ker(\FF_{121}^{\ast}\longrightarrow \FF_{11}^{\ast})$ given by the norm map it is also isomorphic to $\mu_{12}$, the group of $12$-th roots of unity (note that $\mathbb F_{121}\cong \FF_{11}(\mu_{12})$). Therefore, $\Gamma_{\OO}^1/\Gamma_{\OO}^1(\Pi_{11})\cong \mu_6$, because $-1\notin B_{v_{11}}^1(\Pi_{11})$. Let $\zeta_6$ be a primitive sixth root of unity and consider $\Gamma=\{x\in \Gamma_{\OO}^1\mid x\mod\Pi_{11}\in \langle\zeta_6^2\rangle\}$. $\Gamma$ is a subgroup of index two in $\Gamma_{\OO}$ and it is torsion-free, because it does not contain elements of order two. For if an element of order two is in $\Gamma$ then it is already contained in $\Gamma_{\OO}^1(\Pi_{11})$. But the latter group is torsion-free, because the level is not divisible by $2$. Hence $X_{\Gamma}$ is a fake quadric. We note that this example is indicated in \cite{Otsubo85}.\\

\noindent In a similar way an appropriate congruence subgroup will lead to a fake quadric in the case $d_B=v_2v_5$. Here, assume first that $\Gamma$ is contained in $N\Gamma_{\OO}^+$. Then, $\chi(N\Gamma_{\OO}^+)=1/20$ and $\chi(\Gamma_{\OO}^1)=1/5$. There are only torsion elements of order $5$ in $\Gamma_{\OO}^1$, because $\pi_2$ is split in $k(\zeta_3)$ and $\pi_5$ is split in $k(\zeta_4)$. Similar to the previous case, we consider the congruence subgroup $\Gamma_{\OO}^1(\Pi_2)$, where $\Pi_2$ is the ideal of $\OO$ lying over $v_2$. Noting that in this case $\Gamma_{\OO}^1/\Gamma_{\OO}^1(\Pi_2)\cong B^1_{v_2}/B^1_{v_2}(\Pi_2)$ (since in this case $-1\in B_{v_2}^1(\Pi_2)$), we have by Riehm's theorem, $[\Gamma_{\OO}^1:\Gamma_{\OO}^1(\Pi_2)]=5$. Furthermore, $\Gamma_{\OO}^1(\Pi_2)$ is torsion-free, because it cannot contain elements of order $5$. Hence, $X_{\Gamma_{\OO}^1(\Pi_2)}$ is a fake quadric.\\
\noindent Assume now that $\Gamma<\Gamma_{S,\OO}^+$ with $S\neq\emptyset$. By Corollary \ref{d5b}, $S=\{v_3\}$ or $S=\{v_{19}\}$ and we have (see Remark \ref{m}) $\chi(\Gamma_{S,\OO}^+)=1/4$ for $S=\{v_3\}$ and $\chi(\Gamma_{S,\OO}^+)=1/2$ for $S=\{v_{19}\}$. 
Recall that $\Gamma_{\{v\},\OO}^+$ is defined as the normalizer $N\Gamma_{\mathcal E}^+$ of an Eichler order of level $\pP_v$. With this description, the finite group $H(\{v\},B)/k^{\ast^2}$ is isomorphic to the factor group $N\Gamma_{\mathcal E}^+/\Gamma_{\mathcal E}^1$.\\
Let us consider the case $S=\{v_{19}\}$. Here we have to search for a torsion-free subgroup in $\Gamma_{\{v_{19}\},\OO}^+=N\Gamma_{\mathcal E}^+$ of index two. There are no elements of order $m>2$ in $\Gamma_{\{v_{19}\},\OO}^+$. Elements of order two are encoded in the group $H(\{v_{19}\},B)/k^{\ast^2}=\langle [1],[2],[\frac{5+\sqrt{5}}{2}],[\frac{9\pm\sqrt{5}}{2}]\rangle$. Checking all the conditions of Lemma \ref{torsioncrit} (particularly using PARI), we find that $[\pi_5]=[\frac{5+\sqrt{5}}{2}]$ does not lead to a torsion, since in this case the condition $\beta)$ in Lemma \ref{torsioncrit} is not satisfied. On the other hand, all the other elements in $H(\{v_{19}\},B)/k^{\ast^2}$ give rise to a torsion. Hence, any subgroup $\Gamma_{\mathcal E}^1<\Gamma<\Gamma_{\{v_{19}\},\OO}^+$ of index 2 in $\Gamma_{\{v_{19}\},\OO}^+$ contains elements of order two. We conclude that there are no torsion-free subgroups of index two in $\Gamma_{\{v_{19}\},\OO}^+$ at all, because any such must contain $\Gamma_{\mathcal E}^1$.
In the remaining case $S=\{v_3\}$ we are indeed able to produce a lattice of a fake quadric. Namely, consider an Eichler order $\mathcal E=\mathcal E(\pi_3)$ of level $\pi_3$ contained in a maximal order $\OO$. Then $[\Gamma_{\OO}^1:\Gamma_{\mathcal E}^1]=(N\pi_3+1)=10$ and, since $\chi(\Gamma_{\OO}^1)=1/5$, we have $\chi(\Gamma_{\mathcal E}^1)=2$. Investigating again the group $H(\{v_3\},B)/k^{\ast^2}\cong \Gamma_{\{v_3\},\OO}^+/\Gamma_{\mathcal E}^1$, we find that the element $[\pi_3]=[3]\in H(\{v_3\},B)/k^{\ast^2}$ does not lead to a torsion in $\Gamma_{\{v_3\},\OO}^+$, because the ramified prime $\pi_2=2$ is split in $k(\sqrt{-3})$. Let $\gamma_3$ be an element in $\Gamma_{\{v_3\},\OO}^+$ such that $\Nrd(\gamma_3)=3$, then the group $\Gamma=\langle \Gamma_{\mathcal E}^1,\gamma_3\rangle$ generated by $\Gamma_{\mathcal E}^1$ and $\gamma_3$ is a torsion-free lattice in $\Gamma_{\{v_3\},\OO}^+$ containing $\Gamma_{\mathcal E}^1$ with index $2$ and satisfying $\chi(\Gamma)=1$, i.~e.~$X_{\Gamma}$ is a fake quadric.

%Again using Lemma \ref{torsioncrit}, it follows that in the case $S=\{v_{19}\}$, $\Gamma_{S,\OO}^+$ contains an element of order two (by Lemma \ref{torsioncrit}, the class $[2]\in H(S,B)$ leads to an element of order $2$ in $\Gamma_{S,\OO}^+$).  
%Therefore by the criterion (\ref{vielfache}), any torsion-free-subgroup of $\Gamma_{S,\OO}^+$ has an index divisble by $5$. Hence there are no lattices of fake quadrics contained in $\Gamma_{S,\OO}^+$ with $S\neq \emptyset$.\\

\noindent In order to simplify the formulation of final results let us introduce the following notation: 
\begin{definition}
\label{classkdb}
A \textbf{class} $[k,d_B,S,I]$ consists of all torsion-free lattices $\Gamma$ in $\Gamma_{S,\OO}^+$ of index $I$, where $\Gamma_{S,\OO}^+$ is defined by a maximal order $\OO$ in the quaternion algebra $B/k$ unramified at exactly two infinite places and having the reduced discriminant $d_B$. 
\end{definition}
\begin{thm}
\label{alled5}
There are exactly 8 classes $[\QQ(\sqrt{5}),d_B,S,I]$ of irreducible lattices of fake quadrics defined over $\QQ(\sqrt{5})$:

\begin{align*}
&[\QQ(\sqrt{5}),v_2v_{41},\emptyset,2], [\QQ(\sqrt{5}),v_2v'_{41},\emptyset,2]
[\QQ(\sqrt{5}),v_2v_{31},\emptyset,2],[\QQ(\sqrt{5}),v_2v'_{31},\emptyset,2],\\
&[\QQ(\sqrt{5}),v_2v_{11},\emptyset,8], [\QQ(\sqrt{5}),v_2v'_{11},\emptyset,8], [\QQ(\sqrt{5}),v_2v_{5},\emptyset,20], [\QQ(\sqrt{5}),v_2v_{5},\{v_3\},4]\\
\end{align*}

\end{thm}
\begin{proof}
In view of Lemma \ref{d5a} and Corollary \ref{d5b} we only need to show that $d_B=v_2v_3$ and $d_B=v_5v_{11}$ do not lead to lattices of fake quadrics. For this we note that such a lattice must be of index $10$ in $N\Gamma_{\OO}^+$ in the first and of index $6$ in the second case. On the other hand, Lemma \ref{torsioncrit} shows that in the first case there is an element of order three in $N\Gamma_{\OO}^+$, which is in contradiction to (\ref{vielfache}), and also that in the second case there must be an element of order five in $N\Gamma_{\OO}^+$, which implies that torsion-free subgroups of index $6$ are not possible.
\end{proof}
\begin{remark}
Let us note that we provided an example in each class of lattices of fake quadrics defined over $\QQ(\sqrt{5})$. Beside the last one, essentially all the examples have already been identified in \cite{Shavel78} and \cite{Otsubo85}. It should be also noticed that a class may also contain more then one lattice. Already the first examples above show that, for instance, the class $[\QQ(\sqrt{5}),v_2v_{41},\emptyset,2]$ contains two lattices. In order to get the exact number of different fake quadrics, one needs to find a method to list all the conjugacy classes of torsion-free subgroups of a given index $I$ in a maximal arithmetic lattice $\Gamma_{S,\OO}^+$.    
\end{remark}

\subsubsection{Lattices of fake quadrics over $\QQ(\sqrt{2})$}
The commensurability classes $\mathcal C(\QQ(\sqrt{2}),B)$ and lattices which may contain lattices of fake quadrics have been determined in Lemma \ref{d8ab}. As in the previous case, reduced discriminants $d_B=v_2v_{97}$ and $d_B=v_7v_{17}$ do not lead to fake quadrics because in those cases $N\Gamma_{\OO}^+$ is the only possible candidate, but which is not torsion-free. We also remark that examples of lattices of fake quadrics in classes $[\QQ(\sqrt{2}),v_2v_{5},\emptyset,4]$ and $[\QQ(\sqrt{2}),v_3v_{7},\emptyset,2]$ have been found in \cite{Shavel78}. Let us show that there are lattices of fake quadrics in $[\QQ(\sqrt{2}),v_2v_{3},\emptyset,12]$, $[\QQ(\sqrt{2}),v_2v_{7},\emptyset,16]$ and $[\QQ(\sqrt{2}),v_2v_{17},\emptyset,6]$.\\

\noindent In the first case consider the congruence subgroup $\Gamma_{\OO}^1(\Pi_2)$ corresponding to the the prime ideal $\Pi_2$ in $\OO$ lying over $v_2$. By \cite{Riehm70}, this group is of index $3$ in $\Gamma_{\OO}^1$, hence of index $12$ in $N\Gamma_{\OO}^+$, and it is torsion-free, since it does not contain elements of order $3$ (level $\pi_2$ and $3$ are coprime), the only possible torsions in $\Gamma_{\OO}^1$. Since $\chi(\Gamma_{\OO}^1)=1/3$, $X_{\Gamma_{\OO}(\Pi_2)}$ is a fake quadric. \\
In the case $d_B=v_2v_7$ one considers $\Gamma_{\OO}^1(\Pi_7)$, the principal congruence subgroup corresponding to the prime ideal $\Pi_7$ lying over $v_7$. This will be a subgroup of $\Gamma_{\OO}^1$ of index $4$, which is torsion-free (the level is coprime to 2 and 3, the only possible torsions in $\Gamma_{\OO}^1$). \\
Similarly, in the case $d_B=v_2v_{17}$ one considers the group $\Gamma_{\OO}^1(\Pi_2)$, which is of index $3$ in $\Gamma_{\OO}^1$ (again by \cite{Riehm70}) and of index $12$ in $N\Gamma_{\OO}^+$. This group is torsion-free, since it does not contain elements of order $3$, the only possible torsions in $\Gamma_{\OO}^1$. Namely, since $17$ is split in $\QQ(\sqrt{-1})$, $v_{17}$ is split in $k(\sqrt{-1})$ and by Lemma \ref{torsioncrit}, there are no elements of order $2$ in $\Gamma_{\OO}^1$. We need to find a torsion-free extension of $\Gamma_{\OO}^1(\Pi_2)$, containing $\Gamma_{\OO}^1(\Pi_2)$ with index two. For this, observe first that $\Gamma_{\OO}^1(\Pi_2)$ is a normal subgroup of index $12$ in $N\Gamma_{\OO}^+$ and that $N\Gamma_{\OO}^+/\Gamma_{\OO}^1(\Pi_2)$ is abstractly an extension of $ \ZZ/3\ZZ$ by $(\ZZ/2\ZZ)^2$. Here $(\ZZ/2\ZZ)^2\cong N\Gamma_{\OO}^+/\Gamma_{\OO}^1\cong H(B)/k^{\ast^2}$ is generated by the classes $[\pi_2]$ and $[\pi_{17}]$ of totally positive elements $\pi_2=2+\sqrt{2}$ and $\pi_{17}=5+2\sqrt{2}$. Using Lemma \ref{torsioncrit}, we find that $\pi_{17}\in H(B)/k^{\ast^2}$ does not lead to a torsion in $N\Gamma_{\OO}^+$, because (using PARI) we find that $\pi_2$ is split in $k(\sqrt{-\pi_{17}})$. Now, define $\Gamma$ to be the inverse image of $\langle [\pi_{17}]\rangle\subset N\Gamma_{\OO}^+/\Gamma_{\OO}^1(\Pi_{2})$ under the canonical projection. It is an index two torsion-free extension of $\Gamma_{\OO}^1(\Pi_2)$ and therefore a lattice of a fake quadric.   
   
\begin{thm}
\label{alled2}
There are exactly 8 classes $[\QQ(\sqrt{2}),d_B,S,I]$ of lattices of fake quadrics defined over $\QQ(\sqrt{2})$:

\begin{align*}
&[\QQ(\sqrt{2}),v_2v_{3},\emptyset,12], [\QQ(\sqrt{2}),v_2v_{5},\emptyset,2], [\QQ(\sqrt{2}),v_2v_{7},\emptyset,16],[\QQ(\sqrt{2}),v_2v'_{7},\emptyset,16],\\ 
&[\QQ(\sqrt{2}),v_2v_{17},\emptyset,6],[\QQ(\sqrt{2}),v_2v'_{17},\emptyset,6], [\QQ(\sqrt{2}),v_3v_{7},\emptyset,2], [\QQ(\sqrt{2}),v_3v'_{7},\emptyset,2]
\end{align*}

\end{thm}

\subsubsection{Lattices of fake quadrics over $\QQ(\sqrt{3})$}
In the case $d_k=12$, Shavel (\cite{Shavel78}) shows the existence of lattices fake quadrics for $d_B=v_2v_5, v_2v_{13},v_3v_{13}$ contained in $N\Gamma_{\OO}^+$. As before, we can exclude existence of lattices of fake quadrics for $d_B=v_2v_7$ and $d_B=v_3v_5$, because in those cases, the only lattice $\Gamma$ with $\chi(\Gamma)= 1$ is $\Gamma=N\Gamma_{\OO}^+$ which is never torsion-free. According to Lemma \ref{d12ab}, it remains to study the classes $[\QQ(\sqrt{3},v_2v_3,S,I)]$ with $S=\emptyset, \{v_{11}\}, \{v_{23}\}$ and $[\QQ(\sqrt{3}),v_2v_{13},S,I]$ with $S=\{v_{3}\}$. Let us show that $[\QQ(\sqrt{3}),v_2v_3,\emptyset,24]$ is non-empty: Namely, consider the principal congruence subgroup $\Gamma=\Gamma_{\OO}^1(\Pi_2\Pi_3)$ corresponding to the ideal $\Pi_2\Pi_3$ of $\OO$. As the index $[\Gamma_{\OO}^1:\Gamma_{\OO}^1(\Pi_2\Pi_3)]$ of a congruence subgroup is the product of local indices, we have $[\Gamma_{\OO}^1:\Gamma_{\OO}^1(\Pi_2\Pi_3)]=[\Gamma_{\OO}^1:\Gamma_{\OO}^1(\Pi_3)][\Gamma_{\OO}^1:\Gamma_{\OO}^1(\Pi_2)]$. By already cited Theorem of Riehm, we get $[\Gamma_{\OO}^1:\Gamma_{\OO}^1(\Pi_2\Pi_3)]=6$. Additionaly, $[N\Gamma_{\OO}^1:\Gamma_{\OO}^1]=4$, and we need to show that $\Gamma$ is torsion-free. By Lemma \ref{torsioncrit} we know that $\Gamma_{\OO}^1$ contains elements of order $2$ and $3$. But no element of order $2$ is contained in $\Gamma_{\OO}(\Pi_2\Pi_3)$ because such an element would also be contained in the localization $B^{1}_{v_2}(\Pi_2)$ at $v_2$, a pro-3 group, which is impossible. By the same argument, $\Gamma_{\OO}^1(\Pi_2\Pi_3)$ cannot contain elements of order $3$, because such elements would be also contained in the local congruence subgroup $B^1_{v_3}(\Pi_3)$, a pro-2 group. Therefore $\Gamma=\Gamma_{\OO}^1(\Pi_2\Pi_3)$ is a lattice of a fake quadric.

%Let us analyze the class $[\QQ(\sqrt{3},v_2v_{13},\{v_{3}\},2)]$. Here, Lemma \ref{torsioncrit} shows that $\Gamma_{\{v_3\},\OO}^+$ contains an element of order $3$. Here, one observes that $(\zeta_3-1)(\zeta_3^{-1}-1)$ is contained in $\pP_{v_3}=\langle \sqrt{3}\rangle$.

Let us consider the class $[\QQ(\sqrt{3}),v_2v_3,\{v_{11}\},4]$. Observe first that $\chi(\Gamma_{\{v_{11}\},\OO}^+)=1/4$ (particularly using Remark \ref{m}). By Lemma \ref{torsioncrit}, we see that there are no elements of order $3$ and $4$ in $\Gamma_{S,\OO}^+$. For elements of order $2$ one has to analyze the group $H(S,B)/k^{\ast^2}$, which in the concrete case is of order $8$ and is generated by the classes $[\pi_{11}\pi_2]$, $[\pi_{11}\pi_3]$ and $[\pi_{2}\pi_3]$ of the primes $\pi_{11}=1+2\sqrt{3}$ (by choice), $\pi_2=1+\sqrt{3}$ and $\pi_3=\sqrt{3}$. Note that interpreting $\Gamma_{S,\OO}^+$ as the normalizer $N\Gamma_{\mathcal E}^+$ of an Eichler order of level $\pi_{11}$, the group $H(S,B)/k^{\ast^{2}}$ is isomorphic to $N\Gamma_{\mathcal E}^+/\Gamma_{\mathcal E}^1$. Lemma \ref{torsioncrit} implies that, for example, $a=\pi_2\pi_{11}$ does not lead to a torsion in $\Gamma_{S,\OO}^+$, because $\pi_{3}$ is split in $k(\sqrt{-a})$. Then, the inverse image $\Gamma_{[\pi_{11}\pi_2]}$ of the group $\langle [\pi_{11}\pi_2]\rangle<H(S,B)/k^{\ast^2}$ under $\Gamma_{S,\OO}\longrightarrow H(S,B)/k^{\ast^2}$ is torsion-free lattice with $\chi(\Gamma_{[\pi_{11}\pi_2]})=1$. The same remains true if we replace $\pi_{11}$ by the conjugate prime $\pi'_{11}=1-2\sqrt{3}$. 

%For this it is enough to show that $\Gamma_{\{v_{11}\},\OO}^+$ contains elements of order $3$, hence by (\ref{vielfache}) every torsion-free subgroup of $\Gamma_{\{v_{11}\},\OO}^+$ has an index divisble by $3$. The existence of an element of order $3$ follows from Lemma \ref{torsioncrit}, d), since $-11\equiv 1\bmod 3$ and therefore $v_{11}$ is split in $k(\zeta_3)$. For the same reason, $[\QQ(\sqrt{3},v_2v_3,\{v_{23}\},2)]$ is empty.  Therefore, we have
In the same way one investigates the class $[\QQ(\sqrt{3}),v_2v_3,\{v_{23}\},2]$. Here we need to find a torsion-free lattice of index two in $\Gamma^+_{\{v_{23}\},\OO}$. Such a lattice corresponds to an index two subgroup in $H(\{v_{23}\},B)/k^{\ast^2}$, which is of order $8$ and generated by classes mod $k^{\ast^2}$ of elements $\pi_{23}\pi_2, \pi_{23}\pi_3, \pi_2\pi_3$. By Lemma \ref{torsioncrit}, every pair of elements in $H(\{v_{23}\},B)/k^{\ast^2}$ leads to a torsion in $\Gamma^+_{\{v_{23}\},\OO}$. Therefore, $[\QQ(\sqrt{3}),v_2v_3,\{v_{23}\},2]$ is empty.

Finally we study the class $[\QQ(\sqrt{3}),v_2v_{13},\{v_{3}\},2]$ consisting of lattices of fake quadrics contained in $\Gamma_{\{v_{3}\},\OO}^+$ which is a maximal member of the commensurability class $\mathcal C(\QQ(\sqrt{3}),v_2v_{13})$. Since $\pi_{13}=4\pm \sqrt{3}$ and the fundamental unit are totally positive and at the same time $\pi_2$ and $\pi_3$ are not totally positive, the group $H(\{v_{13}\},B)/k^{\ast^2}$ is generated by $[\pi_{13}]$ and $[\pi_2\pi_3]$. Note that $\Gamma^+_{\{v_{3}\},\OO}$ can be seen as the normalizer $N\Gamma_{\mathcal E}^+$ of an Eichler order of level $\pi_3$ and contains the torsion-free subgroup $\Gamma_{\mathcal E}^1$ ($\Gamma_{\mathcal E}^1$ is torsion-free since it is contained in the torsion-free lattice $\Gamma_{\OO}^1$, where $\OO$ is a maximal order). With Lemma \ref{torsioncrit} we can check that the residue class $\pi_{13}\bmod k^{\ast^2}$ produces a torsion in $\Gamma_{\{v_{3}\},\OO}$, whereas $\pi_2\pi_3$ does not lead to a torsion if the place $v_{13}$ is appropriately chosen. Namely, the rational prime $13$ splits in $k(\sqrt{-\pi_2\pi_3})$ as a product of three different primes, and therefore one place of $k$ over $13$ splits in $k(\sqrt{-\pi_2\pi_3})$ and the other remains prime. Then, if the split prime over $13$ is in $R_f(B)$, $[\pi_2\pi_3]\in H(S,B)/k^{\ast^2}$ does not come from a torsion, and there is a torsion-free subgroup of index two in $\Gamma_{\{v_3\},\OO}^+$ which is a lattice of a fake quadric. Thus we have   
 
\begin{thm}
 \label{alled3}
There are 9 classes $[\QQ(\sqrt{3}),d_B,S,I]$ of irreducible lattices of fake quadrics defined over $\QQ(\sqrt{3})$, namely:
\begin{align*}
&[\QQ(\sqrt{3}),v_2v_3,\emptyset,24], [\QQ(\sqrt{3}),v_2v_3,\{v_{11}\},4],[\QQ(\sqrt{3}),v_2v_3,\{v'_{11}\},4], [\QQ(\sqrt{3}),v_2v_5,\emptyset,2],\\
&[\QQ(\sqrt{3}),v_2v_{13},\emptyset,4], [\QQ(\sqrt{3}),v_2v'_{13},\emptyset,4], [\QQ(\sqrt{3}),v_2v_{13},\{v_3\},2], [\QQ(\sqrt{3}),v_3v_{13},\emptyset,2],[\QQ(\sqrt{3}),v_3v'_{13},\emptyset,2]
\end{align*}

\end{thm}

\subsubsection{Classification results for $\mathcal C (\QQ(\sqrt{d}),B)$ with $d>12$}
The remaining commensurability classes which are  given in Lemma \ref{classno1ab} and Lemma \ref{classno>1ab} are investigated along the same lines as the cases above. Let us state the final classification result:
\begin{thm}
\label{alleanderen}
Let $k$ be a real quadratic field with discriminant $d_k>12$. Then, the classes $[k,d_B,S,I]$ of irreducible lattices of fake quadrics  are as follows:
\begin{align*}
& [\QQ(\sqrt{13}),v_2v_3,\emptyset,8],[\QQ(\sqrt{13}),v_2v'_3,\emptyset,8], [\QQ(\sqrt{13}),v_3v'_3,\emptyset,12], [\QQ(\sqrt{13}),v_3v_{13},\emptyset,2],[\QQ(\sqrt{13}),v'_3v_{13},\emptyset,2]\\
& [\QQ(\sqrt{17}),v_2v'_2,\emptyset,24], [\QQ(\sqrt{17}),v_2v_{13},\emptyset,2],[\QQ(\sqrt{17}),v'_2v'_{13},\emptyset,2]\\
& [\QQ(\sqrt{21}),v_2v_5,\emptyset,2],[\QQ(\sqrt{21}),v_2v'_5,\emptyset,2]\\
& [\QQ(\sqrt{24}),v_2v_3,\emptyset,8], [\QQ(\sqrt{24}),v_3v_5,\emptyset,2],[\QQ(\sqrt{24}),v_3v'_5,\emptyset,2]\\
& [\QQ(\sqrt{15}),v_2v_3,\emptyset,4]
\end{align*}
\end{thm}

\begin{remark}
We would like to put a remark on the non-emptyness of the classes of fake quadrics given in Theorem \ref{alleanderen}: From \cite{Shavel78} we know that the classes $[\QQ(\sqrt{13}),v_3v_{13},\emptyset,2]$, $[\QQ(\sqrt{13}),v'_3v_{13},\emptyset,2]$, $[\QQ(\sqrt{17}),v_2v_{13},\emptyset,2]$, $[\QQ(\sqrt{17}),v'_2v'_{13},\emptyset,2]$, $ [\QQ(\sqrt{21}),v_2v_5,\emptyset,2]$, $[\QQ(\sqrt{21}),v_2v'_5,\emptyset,2]$, $[\QQ(\sqrt{24}),v_3v_5,\emptyset,2]$ and $[\QQ(\sqrt{24}),v_3v'_5,\emptyset,2]$ are non-empty. As in \cite{Otsubo85}, one can also show that the class $[\QQ(\sqrt{24}),v_2v_3,\emptyset,8]$ is non-empty, by showing that the congruence subgroup $\Gamma_{\OO}^1(\Pi_3)$ is a lattice of a fake quadric. The same holds for the class $[\QQ(\sqrt{24}),v_2v'_3,\emptyset,8]$ and the corresponding congruence subgroup $\Gamma_{\OO}^1(\Pi_3')$. In the case $[\QQ(\sqrt{15}),v_2v_3,\emptyset,4]$ we find know that $\chi(\Gamma_{\OO}^1)=2$. Furthermore, $\Gamma_{\OO}^1$ is torsion-free. By Theorem \ref{torsioncrit}, we find that, for instance, the totally positive element $3\in H(B)$ does not lead to a torsion in $N\Gamma_{\OO}^+$ as $2$ is totally split in $k(\sqrt{-3})$. Threfore there is an element $\gamma_3\in N\Gamma_{\OO}^+$ such that $\Nrd(\gamma_3)=3$ which is not of finite order and such that $\gamma_3^2\in \Gamma_{\OO}^1$. Therefore, $\Gamma=\langle \Gamma_{\OO}^1,\gamma_3\rangle$ is a torsion-free extension of $\Gamma_{\OO}^1$ containing $\Gamma_{\OO}^1$ with index two and defines a lattice of a fake quadric.  
On the other hand we do not know examples of lattices of fake quadrics which belong to the remaining classes $[\QQ(\sqrt{13}),v_3v'_3,\emptyset,12]$ and $[\QQ(\sqrt{17}),v_2v'_2, \emptyset,24]$. Here, no congruence subgroup seems to lead to a lattice of a fake quadric.  
  
\end{remark}
\bibliography{literatura}
\bibliographystyle{acm}

\end{document}